\documentclass[11pt]{article}
\usepackage{fullpage}
\usepackage{fancyhdr}
\usepackage{amsmath,amssymb,amsthm}
\usepackage{amsmath,amscd}
\usepackage{mathtools}
\usepackage{mathrsfs}
\usepackage{tikz-cd}
\usepackage{authblk}

\usepackage[letterpaper, margin=1in]{geometry}

\newcommand{\Z}{\mathbb{Z}}

\newcommand{\R}{\mathbb{R}}
\newcommand{\Q}{\mathbb{Q}}

\newcommand{\tb}[1]{\widetilde{\bold{ #1 }}}

\newcommand{\G}{Gal}

\newcommand{\Spec}{\text{Spec}}
\newcommand{\Spf}{\text{Spf}}

\newcommand{\Sp}{\text{Sp}}

\newcommand{\Rep}{\textbf{Rep}_{\mathbb{Q}_p}}
\newcommand{\Mphi}[1]{\textbf{M}\Phi^\nabla_{ #1, \sigma}}
\newcommand{\Mphiet}[1]{\textbf{M}\Phi^{et,\nabla}_{ #1, \sigma}}

\newcommand{\ang}[1]{\langle #1 \rangle}

\newcommand{\sheafrig}[2][]{\mathcal{O}_{\mathcal{ #2 }_{#1}^{an}}}
\newcommand{\sheafformal}[2][]{\mathcal{O}_{\mathcal{ #2 }_{#1}}}
\newcommand{\vnaive}[2]{v_{ #2 }^{\text{naive}}( #1 )}

\renewcommand {\bar}{\overline}
\newcommand{\im}{\text{Im}}

\DeclarePairedDelimiter\ceil{\lceil}{\rceil}

\newtheorem{theorem}{Theorem}[section]
\newtheorem{lemma}[theorem]{Lemma}

\newtheorem{definition}[theorem]{Definition}

\newtheorem{proposition}[theorem]{Proposition}

\newtheorem{corollary}[theorem]{Corollary}
\newtheorem{remark}[theorem]{Remark}

\def\frak{\relaxnext@\ifmmode\let\next\frak@\else
 \def\next{\Err@{Use \string\frak\space only in math mode}}\fi\next}
\def\goth{\relaxnext@\ifmmode\let\next\frak@\else
 \def\next{\Err@{Use \string\goth\space only in math mode}}\fi\next}
\def\frak@#1{{\frak@@{#1}}}
\def\frak@@#1{\noaccents@\fam\euffam#1}
\font\tengoth=eufm10
\newfam\gothfam \def\goth{\fam\gothfam\tengoth} \textfont\gothfam=\tengoth

\title{The monodromy of $F$-isocrystals with logarithmic decay}
\author{Joe Kramer-Miller}

\date{}

\begin{document}
\maketitle

\newcommand{\Addresses}{{% additional braces for segregating \footnotesize
  \bigskip
  \footnotesize

  Joe Kramer-Miller, \textsc{Department of Mathematics, University College London,
    Gower Street, London}\par\nopagebreak
  \textit{E-mail address}, Joe Kramer-Miller: \texttt{j.kramer-miller@ucl.ac.uk}

}}

\abstract{ Let U be a smooth geometrically connected affine curve over $\mathbb{F}_p$ with compactification X.  Following Dwork and Katz, a $p$-adic representation $\rho$ of $\pi_1(U)$ corresponds to an $F$-isocrystal.  By work of Tsuzuki and Crew an $F$-isocrystal is overconvergent precisely when $\rho$ has finite monodromy at
 each $x \in X-U$.  However, in practice most F-isocrystals arising geometrically are not overconvergent and have logarithmic growth at singularities 
 (e.g. characters of the Igusa tower over a modular curve).  We give a Galois-theoretic interpretation of these log growth $F$-isocrystals in terms of asymptotic properties of higher ramification groups.  
}

\section{Introduction}
\subsection{The Riemann-Hilbert correspondence in positive characteristic}
The classical Riemann-Hilbert correspondence for a Riemann surface $S$ 
provides an equivalence of categories
between complex representations of $\pi_1(S)$ and holomorphic vector bundles on $S$ with 
flat connection.  In Katz' seminal paper \cite{Katz1} 
on $p$-adic modular forms, he proves 
a$\mod{p}$ analogue of this correspondence.  
The correspondence is roughly as follows:
Fix $X$ to be a proper smooth curve over a finite field $k$ of characteristic $p$
and let $D$ be a finite set of $k$-points in $X$.  Let $Y$ be the curve $X-D$
and let $j:Y \to X$ be the natural open immersion.
We fix formal schemes $\mathcal{X}$ (resp. $\mathcal{Y}$) whose special fibers
are $X$ (resp. $Y$).  The correspondence then gives an equivalence of categories
between $p$-adic representations of $\pi_1(Y)$ and \'etale convergent
$F$-isocrystals: vector bundles with 
connections commuting with a Frobenius on $\mathcal{Y}^{an}$
that satisfy certain convergence conditions (see \cite{Berthelot1} for
a precise definition).  
In \cite{Tsuzuki1} Tsuzuki proves that a representation of $\pi_1(Y)$
has finite monodromy around $D$ if and only if the corresponding convergent 
$F$-isocrystal can be extended to a strict neighborhood of $\mathcal{Y}^{an}$
(i.e. it is overconvergent).  The purpose of this article is to
generalize Tsuzuki's result by explaining the monodromy of a wider class of 
$F$-isocrystals that occur geometrically.

When studying the monodromy around a point $x \in D$, it is significantly 
simpler to work locally.  Let $F$ be the fraction field of the
completion of $\mathcal{O}_{X,x}$ and let $G_F$ be the absolute
Galois group of $F$.  We let $\mathcal{O}_K$ be the Witt vectors of $k$ with fraction
field $K$.  Since $X$ is smooth we know
that $F$ is isomorphic to $k((T))$.  The inclusion $G_F \hookrightarrow \pi_1(Y)$
lets us restrict our attention to representations 
$\rho: G_F \to GL_n(\Q_p)$.  To see what happens when we localize an
$F$-isocrystal we require a couple of definitions:
\[  \mathcal{E} := \Big\{ \sum_{n=-\infty}^\infty a_nT^n  \in K[[T,T^{-1}]] \Big |  
\begin{array}  {l}
v_p(a_i) \text{ is bounded below and } \\
v_p(a_i) \to \infty   \text{ as } i \to -\infty 
\end{array}
\Big \}  \]
\[\mathcal{E}^\dagger := \Big\{ f(T) \in \mathcal{E} \Big | \text{ }f(T)\text{ converges on an
annulus } 0<r< |T|_p < 1 \Big\}. \]
Note that $\mathcal{E}$ is the completion of the stalk at the special point
of $\Spec(\mathcal{O}_K[[T]])$ and $\mathcal{E}^\dagger$ are those functions which converge
on a neighborhood of the special point.  A function $f(T)$ is in $\mathcal{E}^\dagger$
if the valuations of the $a_n$ grow at least linearly as $n\to -\infty$.
If we translate the equivalence of categories
to our local situation we get correspondences due to Fontaine and Tsuzuki
(see \cite{Fontaine1}, \cite{Tsuzuki1}, and Section \ref{section define functors}):

\[ \left\{ 
\begin{array} {c}
\text{Free modules $M$ of rank $d$ over $\mathcal{E}$} \\
\text{with a connection $\nabla$ that commutes} \\
\text{with a unit-root Frobenious}
\end{array} \right\} 
\longleftrightarrow
\left\{
\begin{array}{c}
\text{ Continuous representations} \\
 \rho: G_F \to GL_d(\Q_p)
\end{array}
\right\}
 \]

\[ \left\{ 
\begin{array} {c}
\text{Free modules $M$ of rank $d$ over $\mathcal{E}^\dagger$} \\
\text{with a connection $\nabla$ that commutes} \\
\text{with a unit-root Frobenious}
\end{array} \right\} 
\longleftrightarrow
\left\{
\begin{array}{c}
\text{ Continuous representations} \\
 \rho: G_F \to GL_d(\Q_p) \\
 \text{where the inertia group $I_F$ has finite image}
\end{array}
\right\}.
\]

If $(\rho,V)$ is a $p$-adic $G_F$-representation, we obtain
a $(\phi,\nabla)$-module over $\mathcal{E}$
by $(\widetilde{\mathcal{E}} \otimes_{\Q_p} V)^{G_F}$, where $\widetilde{\mathcal{E}}$
is a $p$-adically complete unramified closure of $\mathcal{E}$ endowed
with a derivative and a Frobenius.  This $(\phi,\nabla)$-module
over $\mathcal{E}$ descends to a $(\phi,\nabla)$-module over $\mathcal{E}^\dagger$
if and only if $\rho$ has finite monodromy.
Let $\bold{e}=(e_1,...,e_d)$ be
a $\Q_p$-basis of $V$ and let $\bold{a}=(a_1,...,a_d)$ be a $\mathcal{E}$-basis
of $(V \otimes \widetilde{\mathcal{E}})^{G_F}$.  Then there is
an invertible matrix of periods $A$ with entries in $\widetilde{\mathcal{E}}$ satisfying
\begin{eqnarray*}\label{period matrix1}
\bold{a}&=&A\bold{e}.
\end{eqnarray*}
We may think of $A^{-1}\bold{a}$ as a bases of solutions to 
the differential equation given by $\nabla$.  One consequence of Tsuzuki's
result is that the entries of $A$ have "linear decay"
if and only if $\rho$ has finite monodromy (actually Tsuzuki's theorem
says more: the entries of $A$ have linear decay and are algebraic over 
$\mathcal{E}^\dagger$).  In the finite monodromy case
there is also a close relation between the ramification breaks of $\rho$
and the radii of convergence of $\nabla$ (see \cite{Kedlaya1} for an overview).
The takeaway is that the decay properties of the differential equation
around a point $x$ in $D$ is closely related to the monodromy of $\rho$ at $x$.
It is natural to ask how to realize this relation when $\rho$ has infinite
monodromy.

\subsection{$F$-isocrystals with logarithmic decay}
Let $f:Z \to Y$ be a proper smooth morphism.  Then $\mathscr{F}=R^i f_* \Q_p$ is
a $\Q_p$-lisse sheaf on $Y$ and therefore corresponds to a $p$-adic 
representation $\rho_{\mathscr{F}}$ of $\pi_1(Y)$.  We call a representation
of $\pi_1(Y)$ arising this way \emph{geometric}.  An \'etale $F$-isocrystal over
$Y$ is
called \emph{geometric} if the corresponding representation of $\pi_1(Y)$ is
geometric.  All geometric \'etale $F$-isocrystals
are the unit-root sub-$F$-isocrystal of an overconvergent $F$-isocrystal
(see \cite{Wan1} for a more precise definition).  
Geometric $F$-isocrystals are generally not overconvergent.  For example,
let $Y$ be the ordinary locus of a modular curve and let $E^{univ}$ be
the universal ordinary elliptic curve over $Y$.  A famous theorem of
Igusa states that the action of $\pi_1(Y)$ on the Tate module of 
$E^{univ}$ has infinite monodromy at every supersingular point (see 
\cite{Igusa1}) and therefore the corresponding $F$-isocrystal
is not overconvergent.  We remark that in this example it is possible to
choose an ``excellent" Frobenius lifting that makes the Frobenius structure
 overconvergent.
Another important example are the hyper-Kloosterman sheaves arising from
the Dwork family of hypersurfaces.  In \cite{Sperber1} Sperber shows that 
the $F$-isocrystals associated to hyper-Kloosterman sheaves cannot
admit an overconvergent Frobenius structure.  

This means that if we want to study all geometric \'etale $F$-isocrystals we
must work in a larger category than the category of overconvergent
\'etale $F$-isocrystals.  The correct category turns out to be
\'etale $F$-isocrystals that have at worst \emph{logarithmic decay} at
each ramified point.  These log-decay $F$-isocrystals were first studied
by Dwork and Sperber in \cite{Dwork-Sperber}, where they studied
the meromorphic continuation of unit-root $L$-functions (actually
Dwork and Sperber only considered log-decay in the Frobenius structure,
but one can show that this condition implies log-decay on the
differential structure).  Dwork and Sperber prove that the
unit-root part of an overconvergent $F$-isocrystal have
log-decay, and therefore that all geometric \'etale $F$-isocrystals
have log-decay.  We remark that the category of geometric $F$-isocrystals
is much smaller than the category of log-decay $F$-isocrystals.  
Wan proves in \cite{Wan1} that the unit-root zeta function associated
to a geometric \'etale $F$-isocrystal has a meromorphic continuation
to the entire $p$-adic plane.  Wan also gives examples in \cite{Wan2}
of \'etale $F$-isocrystals with log-decay whose unit-root zeta function
does not admit a meromorphic continuation.

\subsubsection{A local description of log-decay}
Let us precisely 
describe what log-decay means
when we work locally around $x \in D$.  For $r>0$ consider the ring 
$\mathcal{E}^r \subset \mathcal{E}$ defined by

\[  \mathcal{E}^r := \Big\{ \sum_{n=-\infty}^\infty a_nX^n  \in \mathcal{E} \Big |  
\begin{array}  {l}
 \text{ There exists $c$ such that } \\
v_p(a_n) - \frac{\log_p (-n)}{r} > c   \text{ for } n < 0 
\end{array}
\Big \}.  \]
Let $\rho$ be a $p$-adic representation of $G_F$ and let
$M$ be the corresponding $(\phi,\nabla)$-module over $\mathcal{E}$
given by Fontaine's theory.  Then an $F$-isocrystal has
$r$-log-decay at $x$ if the $M$ descends to a $(\phi,\nabla)$-module
on $\mathcal{E}^r$.  By Theorem \ref{Trivializing log-decay diffeqs}
that this is equivalent to the entries of the period matrix
in equation (\ref{period matrix1}) having a log-decay condition.  
The purpose of this article is
to describe a Galois
theoretic property of $\rho$ that determines when $M$ descends to
a $(\phi,\nabla)$-module over $\mathcal{E}^r$.  This gives us to
a deeper understanding of the monodromy properties of geometric representations
 of $\pi_1(Y)$.  
 
 The Galois-theoretic meaning of the log-decay condition has
 to do with the interaction between the $p$-adic Lie filtration
 and the ramification filtration determined by $\rho$.  We draw
 inspiration by a theorem of Sen (see \cite{Sen1}), which nicely
 explains the interaction of these two filtrations for a local field with
 mixed characteristic. 
  Let $K$ be a local field  and let $L$ be a Galois extension
 of $K$ whose Galois group $G_{L/K}$ is a $p$-adic Lie group.
 The condition that $G_{L/K}$ is a $p$-adic Lie group guarantees
 a filtration 
 \[G_{L/K}=G_{L/K}(0) \supset G_{L/K}(1) \supset G_{L/K}(2) \supset ...,\]
 which satisfies \[\bigcap_n^\infty G_{L/K}(n)=\{0\}~~~\text{and,}\] 
 \[G_{L/K}(n+1)=\{s \in G_{L/K} ~|~ s=x^p ~ \text{for }x \in G_{L/K}(n)\}.\]
 We also have a ramification filtration on $G_{L/K}$ described in
 Section \ref{section ramification}.  
 \begin{theorem} (Sen) Assume that $K$ is a finite extension of $\Q_p$
 with ramification index $e$.  There exists $c>0$ such that
 \[ G^{ne +c} \subset G(n) \subset G^{ne-c}.\]
\end{theorem}

This result completely fails for equal characteristic.  In general
the ramification filtration behaves too erratically to have
a reasonable relationship with the Lie filtration.  For example,
let $G$ be the Galois group of a totally ramified 
$\Z_p$-extension of $K$.  Let $s_0<s_1<s_2<...$ be the breaks
of the ramification filtration of $G$.  When $K$ has characteristic $0$ we
deduce from Sen's theorem that $s_n$ is approximately $ne$.  However
in characteristic $p$ it is possible for any sequence of integers $s_i$
to occur as long as $s_{i+1}\geq ps_i$ (\cite[Proposition 15]{Cherbonnier1}).
The $r$-log-decay condition will correspond to the $\Z_p$ extensions
such that the breaks $s_i$ grow $O(p^{ir})$.  More generally, let
$\rho:G_F \to GL(V)$ be a $p$-adic Galois representation.  Let $L\subset V$
be a lattice stable under the action of $G_F$.  We may regard
$G=\im(G_F \to GL(L))$ as a Galois group for an extension over $F$ that
is also a $p$-adic Lie group.  The lattice $L$ gives a $p$-adic Lie filtration
on $G$ by \[ G(n) = \ker(G \to GL(L/p^nL)).\] 
We then define the category $\Rep^r(G_F)$ to be the $p$-adic representations
$\rho$ of $G_F$ such $G^{p^{nr}c} \subset G(n)$.  Our main local result is

\begin{theorem} \label{main local result} 
There is an equivalence of categories

\[ \left\{ 
\begin{array} {c}
\text{\'etale }(\phi,\nabla)\text{-modules over }\mathcal{E}^r.
\end{array} \right\} 
\longleftrightarrow
\left\{
\begin{array}{c}
\Rep^r(G_F)
\end{array}
\right\},
\]
sending a $p$-adic representation $V$ to $D^r(V)$.  The functor $D^r$
is compatible with the functors $D$ and $D^\dagger$ of Fontaine and
Tsuzuki.  That is, if $V$ is in $\Rep^r(G_F)$ then 
$D^r(V)\otimes_{\mathcal{E}^r}\mathcal{E} = D(V)$ and if 
$V$ is in $\Rep^\dagger(G_F)$ then 
$D^\dagger (V) \otimes_{\mathcal{E}^\dagger} \mathcal{E}^r = D^r(V)$.

\end{theorem}

\subsubsection{A global description of log-decay}
In the final section we introduce a global notion of $r$-log-decay $F$-isocrystals
and relate the $r$-log-decay property to the geometry of representations of
$\pi_1(Y)$.  We are inspired by Berthelot's notion of overconvergent sheaves 
(see \cite[Section 2.1]{Berthelot1}).
Berthelot defines a sheaf of rings $\mathcal{O}_{\mathcal{Y}^{an}}^{\dagger}$ on
$\mathcal{X}^{an}$ whose sections on any strict neighborhood of $\mathcal{Y}^{an}$ are
the analytic functions on $\mathcal{Y}^{an}$ that converge on an outter annulus of the residue disk of each $x$ in $D$.  
An convergent $F$-isocrystal $\mathcal{M}$ on $Y$, which is
a vector bundle on $\mathcal{Y}^{an}$ with a connection and Frobenius satisfying certain convergent conditions, 
is overconvergent if 
it extends to a sheaf of $\mathcal{O}_{\mathcal{Y}^{an}}^{\dagger}$-modules.  By expanding
a convergent $F$-isocrystal 
$\mathcal{M}$ in terms of a parameter of $x \in D$, we are able to recover a local
 $F$-isocrystals over $\mathcal{E}$ (see \cite[Section 4]{Crew3} for
a thorough description of this localization process).  If $\mathcal{M}$ is overconvergent this
local $F$-isocrystal will descent to $\mathcal{E}^\dagger$.  

In order to have a global notion of $r$-log-decay, we define a sheaf of rings 
$\mathcal{O}_{\mathcal{Y}^{an}}^{r}$ on $\mathcal{X}^{an}$.  Similar to Berthelot's
construction, the sections of $\mathcal{O}_{\mathcal{Y}^{an}}^{r}$ on a strict
neighborhood of $\mathcal{Y}^{an}$ are
the analytic functions on $\mathcal{Y}^{an}$ that have $r$-log-decay in the residue
disk of each $x$ in $D$.  We then say that an $F$-isocrystal $\mathcal{M}$ on $Y$ has
$r$-log-decay if $\mathcal{M}$ extends to a sheaf of $\mathcal{O}_{\mathcal{Y}^{an}}^{r}$-modules.
We prove that $\mathcal{M}$ has $r$-log-decay if and only if locally at each $x \in D$ 
the corresponding $F$-isocrystal over $\mathcal{E}$ descends to $\mathcal{E}^r$.
That is, our global definition of $r$-log-decay is the same as having $r$-log-decay
locally at each point.  

We then use the Riemann-Hurwitz-Hasse formula to give a global geometric interpretation
of the $r$-log-decay property.  More precisely, let $\rho:\pi_1(Y)\to GL_d(V)$ be a 
$\Q_p$-representation let $L\subset V$ be a stabilized lattice.  We let $G$ be the image of
$\pi_1(Y)$ and use $L$ to define a $p$-adic Lie filtration 
\[G=G(0) \subset G(1) \subset G(2) \subset ...\]
This gives an \'etale pro-$p$ tower over curves 
\[Y=Y_0 \leftarrow Y_1 \leftarrow Y_2 \leftarrow ... \]
Let $X_n$ be the compactification of $Y_n$.  
The Riemann-Hurwtiz-Hasse formula gives a formula for the genus $g_n$ of $X_n$
in terms of the degree $d_n$ and the different $\delta_{X_n/X}$ of the map $X_n \to X$.  
The
different is determined by the higher ramification groups, which allows us to 
deduce upper bounds for $g_n$
from our local results.  More precisely the $F$-isocrystal
associated to $\rho$ has $r$-log-decay if and only if the genus to
degree ratio $\frac{g_n}{d_n}$ grows $O(p^{nr})$.  
This is the content of Theorem \ref{Main global theorem}.

\subsection{Outline of article}
This article is divided into seven sections.  In Section 
\ref{section log growth rings} we develop some basic
properties of the log-decay ring $\mathcal{E}^r$ and its 
unramified extensions.  Section \ref{section phi nabla modules}
contains an overview of the theory of $(\phi,\nabla)$-modules
over various rings of Laurent series.  We then define the
log-decay period rings $\widetilde{\mathcal{E}}^r$ in Section  \ref{section period rings} and develop some of their basic properties.  In Section
\ref{section define functors} we introduce the functors 
$D$, $D^\dagger$ and $D^r$ using the period rings of the previous section.
The sixth section contains an review of the higher ramification groups
and several auxiliary lemmas on ramification theory.  The proof
of Theorem \ref{main local result} is provided in Section
\ref{section Main result}.  Finally, we return to the global situation in section eight:
we give a precise global definition of log-decay $F$-isocrystals and
we deduce global geometric statements about pro-$p$ towers
of curves corresponding to log-decay $F$-isocrystals.

\subsection{Future work on Frobenius distributions and genus stability}

Let $\rho$ be a $p$-adic representation of $\pi_1(X)$ with finite
monodromy and let $M_\rho$ be the corresponding overconvergent $F$-isocrystal.
Berthelot developed a theory of rigid cohomology, which allows coefficients 
in $M_\rho$.  Crew demonstrated that much of
Deligne's $l$-adic arguments in \cite{Deligne1} 
translate to the $p$-adic setting, assuming certain finiteness
conditions on the rigid cohomology groups (see \cite{Crew1}).  For instance,
Crew was able to prove an equidistrobution theorem for the eigenvalues of the Frobenius analogous to Deligne's $l$-adic Cheboratev density theory. 
The finiteness of rigid cohomology has since been proven by Kedlaya (see \cite{Kedlaya4}) using the $p$-adic monodromy theorem of Andre, Mebkhout, and
Kedlaya (see \cite{Andre1}, \cite{Mebkhout1}, and \cite{Kedlaya2}).  
This picture is much less complete when one considers the larger
category of convergent $F$-isocrystals.  In particular, the equidistrobution of 
Frobenius eigenvalues is known to be false.  Deligne's arguments utilize bounds obtained
from the Lefschetz trace formula and the Euler-Poincare formula.  The latter
depends on higher ramification groups.  It therefore seems likely that
any attempt to study the Frobenius distrobution of an $F$-isocrystal
will rely on a thorough understanding the monodromy of this $F$-isocrystal.  One may
hope that the $r$-log-decay property could measure the failure of Frobenius equidistrobution.
The author is currently investigating this phenomenon.  

Let $M$ be an $F$-isocrystal of rank one with $r$-log-decay.  This
corresponds to a $\Z_p$-tower of curves $X_n$.  By Theorem \ref{Main global theorem}
we know that $g_n$, the genus of $X_n$, is bounded by $cp^{(r+1)n}$.  Recent work
of Kosters and Wan prove that $g_n$ is bounded below by a quadratic 
in $p^n$ (see 
\cite[Corollary 5.3]{Kosters-Wan}).  
In fact they give a precise formula for $g_n$ in terms of 
Artin-Shreier-Witt theory.  In upcoming work we show that
if $M$ is the unit-root part of an overconvergent $F$-isocrystal, then
we may take $r=1$ (this was previously known by Wan and Sperber,
as is mentioned in a remark in \cite{Wan1}, though it does not
appear to be published).  It follows that $g_n$ is bounded above
and below by quadratics in $p^n$.  It is natural to ask: when
is $g_n$ given precisely by a quadratic in $p^n$ for large $n$?  
Such a tower is called \emph{genus-stable} and such towers
have been classified by Kosters and Wan in the context of 
Artin-Shreier-Witt theory.  For example,
the Igusa tower is genus-stable.  One may hope that
any rank one $F$-isocrystal that is the unit-root subspace
 of an overconvergent $F$-isocrystal gives rise to a genus-stable
 tower.  This would imply that any $\Z_p$-tower arising geometrically
 is genus-stable.  In particular, let $A\to X$ be an Abelian variety
 such that the Tate module of $A \times_X Y$ is a rank one $\Z_p$-module.
 Then one would hope that
 the tower of curves obtained by adding torsion points to the
 function field of $X$ is genus-stable.  The author is currently 
 investigating these questions.

\subsection{Acknowledgments}
I would like to thank Daqing Wan for his enthusiasm towards 
this project from its earlier stages.
He explained the importance of $F$-isocrystals with log-decay and pointed me to his
papers on the topic.  I would also like to thank Laurent Berger for encouraging this
project and for sharing his personal notes that helped me resolve some technical Witt
vector issues.  In addition, this work has benefited from conversations with
Kiran Kedlaya, Bryden Cais, Joseph Gunther, Pierre Colmez, and Liang Xiao.

\section{Local log decay rings} \label{section log growth rings}
In this section we will establish several algebraic 
properties about rings of Laurent series 
whose tails decay logarithmically.  
Let $k$ be a finite field and let $K$ be the fraction field of the Witt vectors
of $k$.  We denote by $v_p$ the valuation on $K$ normalized so that $v_p(p)=1$.
Consider the sets of Laurent series:

\[  \mathcal{E}_{T,K} := \Big\{ \sum_{i=-\infty}^\infty a_iT^i  \in K[[T,T^{-1}]] \Big |  
\begin{array}  {l}
v_p(a_i) \text{ is bounded below and } \\
v_p(a_i) \to \infty   \text{ as } i \to -\infty 
\end{array}
\Big \}  \]

\[  \mathcal{E}_{T,K}^r := \Big\{ \sum_{n=-\infty}^\infty a_nT^n  \in  \mathcal{E}_{T,K} \Big |  
\begin{array}  {l}
 \text{ There exists $c$ such that } \\
v_p(a_n) - \frac{\log_p (-n)}{r} \geq c   \text{ for } n < 0 
\end{array}
\Big \}.  \]
When the choice of $T$ and $k$ is unambiguous we will drop
the subscripts and refer to $\mathcal{E}_{T,K}$ (resp
$\mathcal{E}_{T,K}^r$) as $\mathcal{E}$ (resp $\mathcal{E}^r$).  
The ring $\mathcal{E}_{T,K}$ is a field that comes with a discrete valuation
\[ v\big ( \sum a_n T^n \big ) = \inf (v_p(a_n)). \]
The valuation ring $\mathcal{O}_{\mathcal{E}_{T,K}}$ consists of elements in 
$\mathcal{E}_{T,K}$ whose coefficients are integral.  The 
residue field is $F:=k((T))$.  

Following Kedlaya (see \cite[Section 2.3]{Kedlaya2}) 
we will introduce naive partial valuations to help
us keep track of the growth of coefficients.  For $a(T) \in \mathcal{E}_{T,K}$
we define
\[\vnaive{a(T)}{n} = \min_{v_p(a_i)\leq n} \{i\}, \]
where $a(T)=\sum a_iT^i$.
Informally, we may think of $\vnaive{a(T)}{n}$ as the $T$-adic valuation
of the first $n$ terms when we $p$-adically expand $a(T)$.  In particular,
it is possible to write \[a(T) = \sum_{n>>-\infty}^\infty q_n(T)p^n,\]
where $q_n(T) \in \mathcal{O}_K((T))$ has $T$-adic valuation $\vnaive{a(T)}{n}$
and the term with the smallest exponent has a $p$-adic unit coefficient.
These partial valuations satisfy the following inequalities:
\[\vnaive{a(T)+b(T)}{n} \geq \min \{\vnaive{a(T)}{n}, \vnaive{b(T)}{n}\} \]
\[ \vnaive{a(T)b(T)}{n} \geq \min_{i + k = n} 
\{\vnaive{a(T)}{i} + \vnaive{b(T)}{k}\}.\]
Equality is obtained if the minimum is achieved exactly once.
The $r$-log-decay condition translates into a condition about
partial valuations:

\begin{lemma} \label{t-adic expansion} Let $a(T) \in \mathcal{E}_{T,K}$.  
Then $a(T) \in \mathcal{E}_{T,K}^r$ if and only if there exists $d>0$
such that $\vnaive{a(T)}{n} \geq -p^{rn}d$ for each $n$.
\end{lemma}
\begin{proof}
Let $n>0$ and let $N$ be the smallest integer such that $v_p(a_N)\leq n$.   
Then $N = \vnaive{a(T)}{n}$.  Fix $c$ such that
\[v_p(a_k) - \frac{\log_p (-k)}{r} \geq c, \]
for all $k<0$. 
We may plug in $N$ for $k$ in this inequality and solve to get 
\[ \vnaive{a(T)}{n}  \geq -p^{r(n - c)} . \]
\end{proof}

\begin{lemma} $\mathcal{E}_{T,K}^r$ is a field.
\end{lemma}

\begin{proof} The only thing to check is that this set is closed under multiplication
and inverses.
Consider two elements of $a(T),b(T) \in \mathcal{E}_{T,K}^r$.  After multiplying
by powers of $p$ we may assume that both elements are
in $\mathcal{O}_{\mathcal{E}_{T,K}}$.  Similarly we may multiply both
elements by powers of $T$ to ensure that their residue in $F$ is a
unit in $k[[T]]$.  In particular $\vnaive{a(T)}{0} = \vnaive{b(T)}{0}=0$.  
Let $c$ be large enough so that 
\[ \vnaive{a(T)}{n},\vnaive{b(T)}{n} > -cp^{rn}.\]
Then we have
\[ \vnaive{a(T)b(T)}{n} \geq \min_{i+j=n}\{ \vnaive{a(T)}{i} + 
\vnaive{b(T)}{j}\}.\]
When $i$ or $j$ is equal to $n$ then the term in the minimum is at
least $-cp^{rn}$.  Otherwise we have $0<i,j<n$, which gives
\[-cp^{ri} + -cp^{rj} \geq -cp^{rn}. \]
Therefore $a(T)b(T) \in \mathcal{E}_{T,K}^r$.

Since $\mathcal{E}_{T,K}$ is a field, we know that $a(T)$ has an inverse $a^{-1}(T)$
contained in $\mathcal{E}_{T,K}$.  
We only need to check that $a^{-1}(T)$ satisfies the $r$-log-decay condition. 
Once again we
assume that $a(T) \in \mathcal{O}_{\mathcal{E}}^\times$ and that the residue
of $a(T)$ is a unit in $k[[T]]$ by multiplying by some
powers of $p$ and $T$.  In particular $\vnaive{a(T)}{0}=0$.
Let $c>0$ be large enough so that
$\vnaive{a(T)}{n} \geq -cp^{rn}$.  We will prove 
$\vnaive{a^{-1}(T)}{n}\geq -cp^{nr}$ by induction on $n$.  When $n=0$
we find that $\vnaive{a^{-1}(T)}{0}=0$ from the multiplicative
inequality and the fact that $\vnaive{1}{0}=0$.  Now assume
$\vnaive{a^{-1}(T)}{j} \geq -cp^{jr}$ for $j<n$.  
We have
\[\vnaive{1}{n}=0 \geq  \min_{i+j=n} \{\vnaive{a(T)}{i} + \vnaive{a^{-1}(T)}{j}\}.\]
When $i=n$ the term in the minimum function is $-cp^{rn}$
and when $0<i<n$ we have \[-cp^{ri} + -cp^{rj} \geq -cp^{rn}. \]
Therefore if $\vnaive{a^{-1}(T)}{n}< -cp^{rn}$ the minimum would
only be obtained exactly once, which yields a contradiction.

\end{proof}
We make $\mathcal{E}_{T,K}^r$ into a valued field by restricting $v$ 
to $\mathcal{E}_{T,K}^r$ and let we
$\mathcal{O}_{\mathcal{E}_{T,K}^r}$ denote the valuation ring.  Note that
$\mathcal{O}_{\mathcal{E}_{T,K}^r}$ consists of the series in $\mathcal{E}_{T,K}^r$ with
integral coefficients.  The maximal ideal is $p\mathcal{O}_{\mathcal{E}_{T,K}^r}$
 and the residue field is 
$F$.  By a Proposition of Matsuda (see \cite[Proposition 2.2]{Matsuda1})
we see that $(\mathcal{O}_{\mathcal{E}_{T,K}^r},p\mathcal{O}_{\mathcal{E}_{T,K}^r})$
is a Henselian pair.  This allows us to deduce the following Lemma about
unramified extensions of $\mathcal{E}_{T,K}^r$.

\begin{lemma} \label{log extensions}
Let $L$ be a finite extension of $F$ and let $k'$ be the residue field of $L$
(so that $k'$ is a finite extension of $k$).  Let $K'$ be the unramified extension
of $K$ whose residue field is $k'$.  There is a unique unramified
extension $E$ of $\mathcal{E}_{T,K}^r$ whose residue field is $L$.  If
$U \in E$ reduces to a unifomrizing element of $L$ then 
$E= \mathcal{E}_{U,K'}$.
\end{lemma}

Lemma \ref{log extensions} tells us that $\mathcal{E}_{T,K}$ does not depend on the lifting of
a parameter of the residue field.  
Similarly, $K$ only depends on the constants of the residue
field.  This means that $\mathcal{E}_{T,K}$ only depends on the residue field.  We will therefore
often refer to $\mathcal{E}^r_{T,K}$ as $\mathcal{E}^r_F$, to indicate the dependence on the 
residue field.  This will be particularly helpful in Section \ref{section Main result},
where we will consider $\mathcal{E}_{F_n}$ for towers of fields $F_n$ over $F$.

\section{$(\phi,\nabla)$-modules over $\mathcal{E}_T$, $\mathcal{E}_{T,K}^r$ and $\mathcal{E}_T^\dagger$} \label{section phi nabla modules}

Let $R$ be either $\mathcal{E}_T$, $\mathcal{E}_{T,K}^r$ or $\mathcal{E}_T^\dagger$.  
The aim of this section is to define $(\phi,\nabla)$-modules over $R$.
Roughly, these are differential modules over $R$
whose derivation is compatible with a Frobenius semi-linear map.  
\begin{definition}
A ring endomorphism $\sigma$ of $R$ is a \emph{Frobenius} if it
induces the Frobenius morphism on $W(k) \subset R$ and
reduces to the Frobenius morphism modulo $p$ :
\[ \sigma(a) \equiv a^q \mod p. \]
A $\phi$-module for $\sigma$ is an $R$-module $M$
equip with a $\sigma$-semilinear endomorphism $\phi: M \to M$
whose linearization is an isomorphism.  More precisely,
we have $\phi(am)=\sigma(a)\phi(m)$ for $a\in R$
and $\sigma^* \phi: R \otimes_{\sigma} M \to M$
is an isomorphism.  We say that $M$ is \'etale or unit-root
if the slopes, in the sense of Dieudonné-Manin, are all zero
(see \cite[Section 5.2]{Kedlaya2} for a more thorough
explanation).

\end{definition}
A Frobenius $\sigma$ on $\mathcal{E}$ descends to $\mathcal{E}_{T,K}^r$ (resp. $\mathcal{E}_T^\dagger$)
if and only if $\sigma(T) \in \mathcal{E}_{T,K}^r$ (resp. $\mathcal{E}_T^\dagger$).
Some examples are the maps induced by $T \to T^q$ or $T \to (1 + T)^q - 1$.

\begin{definition}
Let $\Omega_R$ be the module of differentials of $R$ over $K$.  In
particular $\Omega_R = R dT$.  We define the $\delta_T: R \to \Omega_R$ to be
the derivative $\frac{d}{dT}$.  A $\nabla$-module over $R$
is an $R$ module $M$ equip with a connection.  That is,
 $M$ comes with a $K$-linear map $\nabla: M \to \Omega_R$
 satisfying the Liebnitz rule: $\nabla(am) = \delta_T(a)m + a\nabla(m)$.
\end{definition}

We may view a $\nabla$-module $M$ over $R$ as a differential equation
over $R$ by considering the equation $\nabla(x)=0$.  If
$M$ is free of rank $d$ we may view it as a first order
differential equation in $d$ variables or a $d$-th order
differential equation in one variable by picking a cyclic vector.
Now we may introduce $(\phi,\nabla)$-modules, which is roughly
an $R$-module with $\phi$ and $\nabla$ structures that
are compatible.

\begin{definition} By abuse of notation, define
$\sigma: \Omega_R \to \Omega_R$ be the map induced 
by pulling back the differential along $\sigma$.  In particular 
\[ \sigma(f(T)dT)=\sigma(f(T))d\sigma(T).\]
A $(\phi, \nabla)$-module
$M$ is an $R$-module that is both a $\phi$-module and a $\nabla$-module
with the following compatibility condition:

\begin{center}
\begin{tikzcd}

M \arrow{d}{\phi} \arrow{r}{\nabla} &
M \otimes \Omega_R \arrow{d}{\phi \otimes \sigma} \\

M \arrow{r}{\nabla} & M \otimes \Omega_R .
\end{tikzcd}
\end{center}
We denote the category of $(\phi, \nabla)$-modules over $R$ by $\Mphi{R}$.
An $(\phi, \nabla)$-module is called \'etale if all of its slopes are $0$.  
We denote the subcategory of $\Mphi{R}$ consisting of \'etale $(\phi, \nabla)$-modules
as $\Mphiet{R}$.
\end{definition}

\section{Period rings with growth properties} \label{section period rings}
\subsection{Large period rings and partial valuations}
Let $\tb{A}$ denote the ring of Witt vectors over 
the completion of $F^{alg}$
and let $\tb{B} = \tb{A}[\frac{1}{p}]$.  We will let $\bold{Frob}$
denote the Frobenius endomorphism of $\tb{A}$.  Each
$x \in \tb{B}$ can be written as \[ \sum_{n\gg -\infty}^\infty [x_n]p^n, \] where
$[x_n]$ is the Teichmuller lift of $x_n \in F^{alg}$.  We 
define maps $w_k: \tb{A} \to \R\cup \infty$
by $w_k(x) = \min_{n\leq k} v_T(x_n)$.  These maps satisfy the following
inequalities:
\begin{eqnarray} 
w_k(x+y) & \geq & \min(w_k(x),w_k(y)) \label{Colmez ineq one: 1} \\
w_k(xy) & \geq & \min_{i+j\leq k} (w_i(x) + w_j(y)) \label{Colmez ineq two: 2}.
\end{eqnarray}
If the minimum is obtained exactly once there is equality.  
For large $k$ these are related to the naive partial valuations $v_k$ 
(see \cite{Kedlaya2}).
We may also define maps $w_k: M_{d\times e}(\tb{A}) \to \R \cup \infty$
that send the matrix $A=(a_{i,j})$ to $\min (w_k(a_{i,j}))$.  Equivalently, if
$A$ has the Teichmuller expansion \[ \sum_{n\gg -\infty}^\infty [A_n]p^n, \]
with $A_n \in M_{d \times e}(F^{alg})$, then $w_k(A) = \min_{n\leq k} v_T(A_n)$,
where $v_T(A_n)$ is the smallest valuation occuring in the entries of $A_n$.  
For
$A \in M_{d \times e}(\tb{A})$  and $B \in M_{e \times f}(\tb{A})$ we have the same inequalities:
\begin{eqnarray*}
w_k(A+B) & \geq & \min(w_k(A),w_k(B))   \\
w_k(AB) & \geq & \min_{i+j\leq k} (w_i(A) + w_j(B)).
\end{eqnarray*}

\subsection{Log decay period rings}
Our
period rings will be subrings of $\tb{B}$ with
growth conditions on the valuations of the Teichmuller
representatives.  
Informally $\tb{B}^r$ is the subset of $\tb{B}$ 
consisting of elements where the growth of $-v_T(x_n)$
is no faster than $p^{rn}$.  Similarly, 
$\tb{B}^\dagger$ is the subset of $\tb{B}$
where the growth of $v_T(x_n)$ is bounded by some linear
function (i.e. the overconvergent elements).    
More
precisely
\[  \tb{B}^r := \Big\{ \sum_{n\gg-\infty}^\infty [x_n]p^n  \in \tb{B} ~ \Big|  ~
\begin{array} {l}
 \text{ There exists $c>0$ such that } \\
v_T(x_n) \geq -p^{nr}c   
\end{array}
\Big \}  \] 

\[  \tb{A}^{r,c} := 
\Big\{ \sum_{n = 0}^\infty [x_n]p^n  \in \tb{B}^r ~ \Big|  ~
p^{nr}c + v_T(x_n) \geq c
\Big \}  \] 

\[  \tb{B}^\dagger :=
\Big\{\sum_{n\gg-\infty}^\infty [x_n]p^n  \in \tb{B}  ~ \Big|  ~ 
\begin{array}  {l}
 \text{ There exists $c>0$ such that } \\
nc + v_T(x_n) \gg -\infty   \text{ as } 
n \to \infty 
\end{array}
\Big \}.  \]
In addition we define $\tb{A}^r=\tb{A}\cap \tb{B}^r$ and
$\tb{A}^\dagger = \tb{A} \cap \tb{B}^\dagger$.  We remark
that $\tb{A}^{r,c} \subset \tb{A}^r$.  It is
not obvious a priori that these sets are subrings of $\tb{B}$.

\begin{lemma} \label{subrings} The sets $\tb{B}^r$, $\tb{A}^r$, $\tb{A}^{r,c}$, $\tb{B}^\dagger$,
and $\tb{A}^\dagger$ are subrings of $\tb{B}$.
\end{lemma}
\begin{proof}
The sets $\tb{B}^\dagger$ and $\tb{A}^\dagger$ are commonplace in $p$-adic
Hodge theory and are known to be rings.  The rest of the assertion
can be proven from the inequalities (\ref{Colmez ineq one: 1}) and
(\ref{Colmez ineq two: 2}).  We provide
a proof for $\tb{A}^{r,c}$ as an example.  Let 
\[x=\sum_{n = 0}^\infty [x_n]p^n ~~ \text{and} ~~ y = \sum_{n = 0}^\infty [y_n]p^n\]
 be two elements
of $\tb{A}^{r,c}$.  Inequality (\ref{Colmez ineq one: 1}) immediately gives $x+y\in\tb{A}^{r,c}$
so we need to show, so we need to show that $xy \in \tb{A}^{r,c}$.
Our definition of $\tb{A}^{r,c}$ implies that
$w_n(x),w_n(y)\geq c - p^{nr}c$ for all $n$.
Let $i,j\geq 0$ be whole numbers such that $i+j\geq n$ and 
$w_n(xy)\geq w_i(x)+w_j(y)$.  Note that $p^{nr}+1\geq p^{ir}+p^{jr}$.
This together with inequality (\ref{Colmez ineq two: 2}) gives
\begin{eqnarray*}
w_n(xy) & \geq & w_i(x)+w_j(y) \\
& \geq & 2c - c(p^{ir} - p^{jr}) \\
& \geq & 2c - c(p^{nr} + 1) \\
& = & c - p^{nr}c,
\end{eqnarray*}
which finishes our claim.
\end{proof}

\begin{lemma} \label{Frobenius lemma} The Frobenius map $\bold{Frob}$
on $\tb{B}$ induces isomorphism on $\tb{B}^\dagger, \tb{A}^\dagger,
\tb{B}^r$, and $\tb{A}^r$.  Furthermore, it induces an isomorphism
\[ \bold{Frob}: \tb{A}^{r,c} \stackrel{\sim}{\longrightarrow} \tb{A}^{r,pc}. \]
\end{lemma}

\begin{proof} Let $x= \sum_{n=0}^\infty [x_n]p^n \in \tb{A}^{r,c}$.  Then
$\bold{Frob}(x)= \sum_{n=0}^\infty[x_n^p]p^n$.  We find
\begin{eqnarray*}
v_T(x_n^p)-(cp)p^{nr} & = & p(v_T(x_n) - cp^{nr}) \\
& \geq & pc,
\end{eqnarray*}
so that $\bold{Frob}: \tb{A}^{r,c} \hookrightarrow \tb{A}^{r,pc}$.  
As $F^{alg}$ is perfect we may consider $\bold{Frob}^{-1}$ and
similar inequalities show $\bold{Frob}^{-1}: \tb{A}^{r,pc} \hookrightarrow \tb{A}^{r,c}$.
This proves the claim for $\tb{A}^{r,c}$.  The
proof for the other rings are similar and are standard in the overconvergent
case (see \cite[I.3]{Berger}).

\end{proof}

\begin{theorem} \label{Embedding Theorem} 
Let $\sigma$ be a Frobenius endomorphism of $\mathcal{E}^r$.  
Then there exists a unique embedding 
\[ i_\sigma: \mathcal{E}^r \hookrightarrow \tb{B}^r\]
satisfying:
\begin{itemize}
\item $i_\sigma$ is a morphism of $K$-algebras
\item $i_\sigma$ induces the inclusion $F \hookrightarrow F^{alg}$ on residue fields
\item $i_\sigma$ commutes with Frobenius (i.e. 
$i_\sigma \circ \sigma = \bold{Frob} \circ i_\sigma$)
\end{itemize}
\end{theorem}

\begin{proof}
The existence and uniqueness of $i_\sigma: \mathcal{E}^r \hookrightarrow \tb{B}$ 
satisfying the desired properties is described by Tsuzuki
(see \cite[Lemma 2.5.1]{Tsuzuki1}).  Therefore it suffices
to prove that the image of $i_\sigma$ lies in $\tb{B}^r$.  Let $T$ be a local
parameter of $\mathcal{E}^r$, so that $\mathcal{E}^r = \mathcal{E}_{K,T}^r$,
and let $P(T)=\sigma(T)$.  
Then $v=i_\sigma(T)$ satisfies the equation $\bold{Frob}(x)=P(x)$.  The theorem
will follow from the following Lemma.
\end{proof}
\begin{lemma} Let $v \in \tb{A}$ be a solution of $\bold{Frob}(x)=P(x)$ with $w_0(v)=1$.  Then
$v \in \tb{A}^r$.
\end{lemma}
\begin{proof}
Write $P(T)=\sum_{-\infty}^\infty a_n T^n$.  Note that $a_p=1$ and that
$p|a_n$ whenever $n\neq p$.  There exists $c \in \R$
such that
\[v_p(a_n) - \frac{\log_p(-n)}{r} > c,\]  for all $n<0$.  
Choose $d$ to be large enough so that
$ d> \frac{1}{p^r-1}$, $pd> \frac{1}{p^{rc} - 1}$, and 
$pd>\frac{p^{r-cr}}{p^r-1}$.

We claim that
for any $x \in \tb{A}^{r,d} + p^k\tb{A}$ with $w_0(x) = 1$
we have $P(x) \in \tb{A}^{r,pd} + p^{k+1}\tb{A}$.
To do this, it is enough to show that $a_nx^n \in \tb{A}^{r,pd} + p^{k+1}\tb{A}$ for each
$n$.  When $n=p$ we find that $x^p \in \tb{A}^{r,pd} + p^{k+1}\tb{A}$ by
checking the binomial expansion.  When $n\geq 0$ and $n\neq p$ we have $p|a_n$,
so we must have $a_nx^n \in \tb{A}^{r,pd} + p^{k+1}\tb{A}$.  The difficult part
is dealing with negative exponents of $x$.

Write $x = [x_0] + [x_1]p + ...+ [x_{k-1}] p^{k-1} + x'p^k$.  For $n<0$
we know that $e= \max (c + \frac{\log_p (-n)}{r}, 1)$ is less
than or equal to $v_p(a_n)$.  
In particular if $p^e x^n$ is contained in $\tb{A}^{r,pd} + p^{k+1}\tb{A}$ 
then so is $a_nx^n$.  We find that
\[x^{-1} = [x_0]
\frac{1}{1 + [\frac{x_1}{x_0}]p + ... + [\frac{x_{k-1}}{x_0}]p^{k-1} + x'p^k}. \]
Expanding this geometric sum we see that $p^e x^n$ consists of terms of the form
\[ p^e[x_0^n]\prod_{i=1}^k \big([\frac{x_i}{x_0}]p^i \big )^{n_i}. \]
Thus it is enough to show that $[\frac{x_i}{x_0}]p^i$ and $p^e[x_0^n]$
are contained in $\tb{A}^{r,pd}$.  

To prove that $[\frac{x_i}{x_0}]p^i \in \tb{A}^{r,pd}$ it suffices to show
\[v_T(\frac{x_i}{x_0}) = v_T(x_i)-1> pd - p^{ir}pd. \tag{A} \]  As $d>\frac{1}{p^r-1}$
we see that $d>\frac{1}{p^{ir}-1}$ and therefore \[(p-1)(p^{ir}-1)d>1 \tag{B}. \]
Subtracting $(B)$ from the inequality $v_T(x_i)>d-p^{ir}d$ yields $(A)$.  To show
$p^e[x_0^n] \in \tb{A}^{r,pd}$ we must prove the inequality 
\[ v_T(x_0^n) = n > pd - p^{re}pd \tag{C}. \]  
When $c + \frac{\log_p (-n)}{r}\geq 1$,
our definition of $e$ allows us to rewrite $(C)$ as
\[n - p^{rc}npd > pd. \tag{C1} \]  
From $pd>\frac{1}{p^{rc}-1}$ we see that
$p^{rc}pd-1>pd$.  Since $-n$ is positive we then have
$-n(p^{rc}pd-1)>pd$, which gives $(C1)$.  If
$c + \frac{\log_p (-n)}{r} <  1$, then $e=1$ and we can rewrite $(C)$ as 
\[ n +p^r pd> pd \tag{C2}. \]  
As $pd>\frac{p^{r-cr}}{p^r-1}$ we see that
\[ -p^{r-cr}+p^r pd > pd. \tag{D} \]  The condition $c + \frac{\log_p (-n)}{r} <  1$
implies $n> -p^{r-cr}$, which combines with $(D)$ to give $(C2)$.  This concludes
the proof of $P(x) \in \tb{A}^{r,pd} + p^{k+1}\tb{A}$.

We know that $v_T(v_0) = 1$ and $v \in \tb{A}^{r,d} + p\tb{A}$.  Proceeding
inductively, we will assume that $v \in \tb{A}^{r,d} + p^k\tb{A}$.  
By the paragraphs above we have
$\bold{Frob}(v) = P(v) \in \tb{A}^{r,pd} + p^{k+1}\tb{A}$.  Then Lemma
\ref{Frobenius lemma} tells us that
$v \in \tb{A}^{r,d} + p^{k+1}\tb{A}$.  It follows that $v \in \tb{A}^{r,d} \subset \tb{B}^r$.

\end{proof}

Let $\sigma$ be a Frobenius of $\mathcal{E}^r$.  We may view $\mathcal{E}^r$
and $\mathcal{E}$ as subrings of $\tb{B}$ through $i_\sigma$.  By Theorem
\ref{Embedding Theorem} we know that $\mathcal{E}^r$ is actually
a subring of $\tb{B}^r$.  A natural question is whether $\tb{B}^r$
contains any element of $\mathcal{E}$ that does not have $r$-log-decay.
More succicently: is $\mathcal{E}\cap \tb{B}^r$ equal to $\mathcal{E}^r$?
This answer is affirmative by Corollary \ref{naive log vs witt log}.
To prove this fact we introduce an auxilary subring of 
$\mathcal{O}_{\mathcal{E}^r}$ defined in an analogous manner
as $\tb{A}^{r,c}$.

\[  \mathcal{O}_{\mathcal{E}^{r,c}} := 
\Big\{ f(T) \in \mathcal{O}_\mathcal{E} \Big|  ~
 \vnaive{f(T)}{n} \geq c - p^{nr}c
\Big \}  \]

\begin{lemma} \label{partial valuations of powers of T}
 Let $c>0$ be large enough so that $T[T^{-1}] \in \tb{A}^{r,c}$.  Then 
 \[w_n(T^k) \geq c - p^{rn}c + k \]
 for all $k$.
 
\end{lemma}
\begin{proof}
Let $\alpha = T[T^{-1}]$.  First we prove that $w_n(\alpha^k) \geq c - p^{rn}c$ for any $k\geq 1$ by induction
on $k$.  
When $k=1$ the inequality is true because $\alpha \in \tb{A}^{r,c}$.  For $k>1$ we have 
the inequality
\[ w_n(\alpha^k) \geq \min_{i + j = n} \{ w_i(\alpha^{k-1}) + w_j(\alpha)\}.\]
When $i=n$ and $j=0$ the term in the minimum is greater than $c - p^{rn}c$ by
our inductive hypothesis and since $w_0(\alpha)=0$.  Similarly for $i=0$
and $j=n$.  For $0<i,j<n$ the term in the minimum is greater than
$2c - p^{ri}c - p^{rj}c$, which is greater than $c-p^{rn}c$.  It follows
that $w_n(\alpha^k)\geq c - p^{rn}c$.  For negative exponents, we will prove
that $w_n(\alpha^{-k})\geq c - p^{rn}c$ by induction on $n$.  For $n=0$ we know that
$w_0(\alpha^{-k})=0$.  Now consider the inequality
\[ 0 = w_n(\alpha^k \alpha^{-k}) \geq \min_{i + j = n} \{ w_i(\alpha^k) + w_j(\alpha^{-k})\}.\]
Using our inductive hypothesis 
we know that for $i>0$ each term in the minimum is
greater than or equal to $c-p^{rn}c$.  For $i=0$ we know that $w_0(\alpha^k)=0$.  Therefore
if $w_n(\alpha^{-k})$ is less than $c-p^{rn}c$, the minimum is less than $c-p^{rn}c$.  This
minimum is acheived exactly once, meaning the inequality is actually an equality.  
This gives a contradiction.  
We now see that for any $k$ we have
\[ w_n(T^k) = w_n(\alpha^k [T^k]) =w_n(\alpha^k)+k,\]
which gives the desired inequality.

\end{proof}

\begin{proposition} \label{witt rc vs naive rc} Let $c>0$ be large enough so that $T[T^{-1}] \in \tb{A}^{r,c}$.  Then
 \[ \tb{A}^{r,c} \cap \mathcal{O}_{\mathcal{E}} = \mathcal{O}_{\mathcal{E}^{r,c}}.\]
\end{proposition}

\begin{proof} Let $f(T) \in \tb{A}^{r,c} \cap \mathcal{O}_{\mathcal{E}}$ and let
 $a_n = \vnaive{f(T)}{n}$.  We may write \[ f(T) = \sum_{n=0}^\infty q_n(T)p^n, \]
 where $q_n(T)=u_n(T)T^{a_n}$ and $u_n(T)$ is a unit in $\mathcal{O}_K[[T]]$.  Define 
 the partial $p$-adic sums \[s_n(T) = \sum_{i=0}^n q_i(T)p^i.\]  
 By induction on $n$ we will simultaniously prove that $a_n \geq c - p^{rn}c$ and that
 the partial sum $s_n(T)$ is contained in $\tb{A}^{r,c}$.
 When $n=0$ there is nothing to check.  For $n>0$ we have
 \[w_{n+1}(s_{n+1}(T)) = w_{n+1}(f(T)) \geq c - p^{r(n+1)}c, \]
 since $f(T) \in \tb{A}^{r,c}$.  Also, by our inductive hypothesis we know that
 $s_n(T) \in \tb{A}^{r,c}$, which means 
 \[w_{n+1}(s_n(T)) \geq c - p^{r(n+1)}c.\]
 By the inequality (\ref{Colmez ineq one: 1}) we know 
 \[w_{n+1}(s_{n+1}(T)) \geq \min \{ w_{n+1}(s_n(T)), w_{n+1}(q_{n+1}(T)p^{n+1})\}.\]
 If $w_{n+1}(q_{n+1}(T)p^{n+1})< c-p^{r(n+1)}c$ the above is an equality, which yeilds
 a contradiction.  Since 
 $a_{n+1} = w_0(q_{n+1}(T)) = w_{n+1}(q_{n+1}(T)p^{n+1})$ we see that 
 $a_{n+1} \geq c - p^{r(n+1)}c$.  
 
 It remains to prove $s_{n+1}(T) \in \tb{A}^{r,c}$.  Since $s_{n}(T)$ and
 $u_n(T)$ are both contained in $\tb{A}^{r,c}$ it is enough to prove
 $T^{a_{n+1}}p^{n+1} \in \tb{A}^{r,c}$.  For $k< n+1$ we have
 $w_k(T^{a_{n+1}}p^{n+1})=\infty$.  
 By Lemma \ref{partial valuations of powers of T}
 we have for $k\geq0$
 
 \begin{eqnarray*}
  w_{n+1+k}(T^{a_{n+1}}p^{n+1}) & = & w_k(T^{a_{n+1}}) \\
  & \geq & c - p^{rk}c + a_{n+1} \\
  & \geq & c - p^{rk}c + c - p^{r(n+1)}c \\
  & \geq & c - p^{r(n+1+k)}c,
 \end{eqnarray*}
which gives $T^{a_{n+1}}p^{n+1} \in \tb{A}^{r,c}$.  It follows that 
 $\tb{A}^{r,c} \cap \mathcal{O}_{\mathcal{E}} \subset \mathcal{O}_{\mathcal{E}^{r,c}}$.

 Conversely, let $f(T)$ be in $\mathcal{O}_{\mathcal{E}^{r,c}}$ and let $a_n$ be $\vnaive{f(T)}{n}$.  We need to
 show $f(T) \in \tb{A}^{r,c}$.  Just as before we may write
 \[ f(T) = \sum_{n=0}^\infty q_n(T)p^n,\]
 where $q_n(T)=u_n(T)T^{a_n}$ and $u_n(T)$ is a unit in $\mathcal{O}_K[[T]]$.  Since $\tb{A}^{r,c}$ is complete with respect to
 the $p$-adic topology, it will suffice to prove that $q_n(T)p^n\in \tb{A}^{r,c}$
 for each $n$.  We also know that $u_n(T) \in \tb{A}^{r,c}$, so we just need to
 show $T^{a_n}p^n \in \tb{A}^{r,c}$.  Since $a_n \geq c - p^{rn}c$ we see from Lemma 
 \ref{partial valuations of powers of T} 
 \begin{eqnarray*}
  w_{k+n}(T^{a_n}p^n) & = & w_k(T^{a_n}) \\
  & \geq & c - p^{rk}c + c - p^{rn}c \\
  & \geq & c - p^{r(k+n)}c.
 \end{eqnarray*}

\end{proof}

\begin{corollary} \label{naive log vs witt log}
 Let $c>0$ be large enough so that $T[T^{-1}] \in \tb{A}^{r,c}$.  Then
 \[ \tb{B}^r \cap \mathcal{E} = \mathcal{E}^r.\]
\end{corollary}
\begin{proof} This follows from Proposition \ref{witt rc vs naive rc} combined with
 the fact that $\tb{A}^{r,c}[\frac{1}{p}] = \tb{B}^r$ and 
 $\mathcal{O}_{\mathcal{E}^{r,c}}[\frac{1}{p}] = \mathcal{E}^r$.
\end{proof}

Let $\mathcal{E}^{r'}$ be an unramified extension of $\mathcal{E}^r$ with
a Frobenius endomorphism $\sigma$ extending that on $\mathcal{E}^r$.  By
Theorem \ref{Embedding Theorem} we may extend $i_\sigma$ to an embedding of
$\mathcal{E}^{r'}$ in $\tb{B}^r$.  In particular, we may embed the maximal
unramified extension of $\mathcal{E}^r$ into $\tb{B}^r$ in a Frobenius
compatible way.  We let $\widetilde{\mathcal{E}}^r$ denote the 
$p$-adic completion
of the maximal extension of $\mathcal{E}^r$ inside of $\tb{B}^r$ and
we let $\mathcal{O}_{\widetilde{\mathcal{E}}^r} = \tb{A} \cap \widetilde{\mathcal{E}}^r$.  Similarly we define $\widetilde{\mathcal{E}}$ to 
be the $p$-adic completion of the maximal unramified extension of $\mathcal{E}$ in $\tb{B}$ and we
let $\mathcal{O}_{\widetilde{\mathcal{E}}} = \tb{A} \cap \widetilde{\mathcal{E}}$.  Note that
$\widetilde{\mathcal{E}}^r \subset \widetilde{\mathcal{E}}$.  There is an action
of $G_F$ on $\widetilde{\mathcal{E}}^r$, 
$\mathcal{O}_{\widetilde{\mathcal{E}}^r}$, $\widetilde{\mathcal{E}}$, and $\mathcal{O}_{\widetilde{\mathcal{E}}}$.  
The invariants are 
$(\widetilde{\mathcal{E}}^r)^{G_F} = \mathcal{E}^{r}$,
$\mathcal{O}_{\widetilde{\mathcal{E}}^r}^{G_F} = \mathcal{O}_{\mathcal{E}^{r}}$, $\widetilde{\mathcal{E}}^{G_F} = \mathcal{E}$ and
$\mathcal{O}_{\widetilde{\mathcal{E}}}^{G_F} = \mathcal{O}_{\mathcal{E}}$.
For a finite separable extension $K$ over $F$
we have $\mathcal{O}_{\mathcal{E}_{K}^{r,c}} = \mathcal{E}_K \cap \tb{A}^{r,c}$ by
Proposition \ref{witt rc vs naive rc}.
We define 
$\mathcal{O}_{\widetilde{\mathcal{E}}^{r,c}}=\widetilde{\mathcal{E}} \cap 
\tb{A}^{r,c}$.  Then $\mathcal{O}_{\widetilde{\mathcal{E}}^{r,c}}$ is the $p$-adic
completion of the $\cup_{K/F}\mathcal{O}_{\mathcal{E}_{K}^{r,c}}$, where the union
is taken over all finite separable extensions $K$ over $F$.

\begin{remark} The analogous result to Theorem \ref{Embedding Theorem} holds
for overconvergent Frobenius endomorphisms.  This allows us to define
a period ring $\bold{B}^\dagger$, which is the $p$-adic completion
of the maximal unramified extension of $\mathcal{E}^\dagger$ in $\tb{B}$.  When $\sigma(T)=(1+T)^p - 1$ then
$\bold{B}^\dagger$ is the same period ring that arises in $p$-adic Hodge theory.  
The ring $\bold{B}^\dagger$ is strictly bigger than the period
ring $\widetilde{\mathcal{E}}^\dagger$ used by Tsuzuki.  The ring $\widetilde{\mathcal{E}}^\dagger$
is the maximal unramified extension of $\mathcal{E}^\dagger$.  The ring
$\bold{B}^\dagger$ contains elements transcendental over $\mathcal{E}^\dagger$,
while $\widetilde{\mathcal{E}}^\dagger$ consists only of algebraic elements.  
Tsuzuki proves that all $(\phi, \nabla)$-modules over $\mathcal{E}^\dagger$ 
become trivial after a finite extension (i.e. they are quasi-constant),
which means only algebraic elements are needed.  In contrast,
the $(\phi,\nabla)$-modules we will consider only trivialize after
adding transcendental periods.

\end{remark}

\subsection{Some period modules}
In the proof of Theorem \ref{main local result} we will need to utilize some \emph{period submodules} of 
$\widetilde{\mathcal{E}}^{r}$ with very precise growth properties.
Recall from the proof of Theorem \ref{Embedding Theorem} that
there exists $c_0>0$ such that $T \in \tb{A}^{r,c_0}$.
For $c>c_0$ and $n>0$ we define 
\[  \widetilde{\mathcal{E}}^{r,c,n} := 
\Bigg\{ \sum_{i = 0}^\infty [x_i]p^i  \in \widetilde{\mathcal{E}} ~ \Big|  ~
\begin{array} {l}
p^{ir}c + v_T(x_i) \geq c ~~ \text{ for }i<n \\
p^{(n-1)r}c + |w_1(T)| + v_T(x_n) \geq c
\end{array}
\Bigg \}.  \] 
Consider an element $\sum [x_i]p^i$ of 
$\widetilde{\mathcal{E}}^{r,c,n}$.  The
first $n$ terms ($i<n$) will look like an arbitrary element of $\mathcal{O}_{\widetilde{\mathcal{E}}^{r,c}}$ (i.e.
$\widetilde{\mathcal{E}}^{r,c,n}/p^n=\mathcal{O}_{\widetilde{\mathcal{E}}^{r,c}}/p^n$).  
For $i>n+1$ the valuations of $x_i$ have no restrictions.  However
the bound on $v_T(x_n)$ is essentially the same as the bound
on $v_T(x_{n-1})$ up to a constant that does not depend on $n$.
The slight difference in these bounds is to allow $T$ to act
on $\widetilde{\mathcal{E}}^{r,c,n}$ via multiplication.  

\begin{lemma} \label{action of $T$}
Let $c > c_0 \frac{p^{2r} - 1}{p^{r} -1}$.  Then multiplication
by $T$ sends $\widetilde{\mathcal{E}}^{r,c,n}$ to
$\widetilde{\mathcal{E}}^{r,c,n}$
\end{lemma}
\begin{proof} 
Let $x=\sum [x_i]p^i \in \widetilde{\mathcal{E}}^{r,c,n}$ and write
$y=Tx=\sum [y_i]p^i$.  Since 
$T$ and  $x$ are contained in $\mathcal{O}_{\widetilde{\mathcal{E}}^{r,c}} + p^n\mathcal{O}_{\widetilde{\mathcal{E}}}$
we know $y \in \mathcal{O}_{\widetilde{\mathcal{E}}^{r,c}} + p^n\mathcal{O}_{\widetilde{\mathcal{E}}}$.
This leaves us with checking $v_T(y_n) \geq c - p^{(n-1)r}c + w_1(T)$.
By inequality (\ref{Colmez ineq two: 2}) we
have \[v_T(y_n) \geq \min_{i+j=n} \{w_i(x) + w_j(T)\}\] so it
will suffice to prove 
\[ w_i(x)+w_j(T)\geq c - p^{(n-1)r}c - w_1(T), \] whenever
$i+j=n$.  We have $w_n(x)\geq c - p^{(n-1)r}c + w_1(T)$ since
$x \in \widetilde{\mathcal{E}}^{r,c,n} $ and $w_0(T) = 1$, which proves
the inequality for $i=n$ and $j=0$.  Similarly we have
$w_{n-1}(x) \geq c - p^{(n-1)r}c$, which gives the inequality for
$i=n-1$ and $j=1$.  Finally we assume that $1<i,j<n-1$.  The inequality in
the hypothesis gives 

\[c_0- p^{jr}c_0>c - p^{(j-1)r}c.\]
Then since $w_i(x) \geq c - p^{ir}c$ and $T \in \tb{A}^{r,c_0}$ we have
\begin{eqnarray*}
 w_i(x) + w_j(T) & \geq &   c - p^{ir}c + c - p^{(j-1)r}c\\
 & \geq & c - p^{nr}c.
\end{eqnarray*}

\end{proof}

For a finite extension $K$ over $F$ we define 
$\mathcal{E}_K^{r,c,n} = \widetilde{\mathcal{E}}^{r,c,n} \cap 
(\mathcal{O}_{\mathcal{E}_K}+ p^{n+1}\mathcal{O}_\mathcal{E})$.  Informally
we may think of $\mathcal{E}_K^{r,c,n}$ as elements in $\widetilde{\mathcal{E}}^{r,c,n}$ that look like elements of 
$\mathcal{O}_{\mathcal{E}_K}$ when reduced modulo $p^{n+1}$.  In particular, 
there is an action of $G_{K/F}$
on $\mathcal{E}_K^{r,c,n}/p^{n+1}$.

\subsection{The functors $D_{\mathcal{E}}$, $D_{\mathcal{E}^\dagger}$, and $D_{\mathcal{E}^r}$}  \label{section define functors}

Let $R$ be one of $\mathcal{E}$, $\mathcal{E}^\dagger$, or $\mathcal{E}^r$
all with residue field $F$.  Fix a Frobenius $\sigma$ on $R$.  Denote by $\Rep(G_F)$ the category of
continuous $\Q_p$ representations of $\G(F^{sep}/F)$ and let $\Rep^{fin}(G_F)$
be the subcategory of representations where the image of the inertia group is finite.
In this section we will define a functor $D_R$ from the $\Rep(G_F)$ to $\Mphi{R}$
and we describe results of Fontaine and Tsuzuki about these functors.

\begin{definition} 

Let $V \in \Rep(G_F)$.  Define $D_R(V)$ to be $(\widetilde{R} \otimes V)^{G_F}$.
The connection on $D_R(V)$ is given by $\nabla=\delta_T \otimes 1$ and the Frobenius
semi-linear morphism is given by $\phi=\sigma \otimes 1$.  We say that
$V \in \Rep(G_F)$ is $\widetilde{R}$-admissible if $\dim_K V = \dim_R D_R(V)$.

\end{definition}

\begin{theorem} \label{Fontaine-Tsuzuki} The functor \[D_\mathcal{E}: \Rep(G_F) \to \Mphiet{\mathcal{E}}\]
is an equivalence of categories.  The functor 
\[D_{\mathcal{E}^\dagger}: \Rep^{fin}(G_F) \to \Mphiet{\mathcal{E}^\dagger}\]
is an equivalence of categories.  In other words, all Galois representations
are $\widetilde{\mathcal{E}}$-admissible and then $\widetilde{\mathcal{E}}^\dagger$-admissible
Galois representations are those with finite monodromy.
\end{theorem}
\begin{proof} The first statement is due to Fontaine (\cite{Fontaine1}) and the second statement is due to 
Tsuzuki (\cite{Tsuzuki1}).
\end{proof}

\section{Ramification theory} \label{section ramification}

\subsection{The higher ramification groups}
We first
recall the definition and basic properties of the higher ramification groups.
Let $L$ be a separable extension of $F=k((T))$ such that $G_{L/F}$ is a finite
dimensional $p$-adic Lie group.  Let $v$ denote the $T$-adic valuation on $F$ normalized so that
$v(T)=1$.  
When $L$ is a finite
extension we let $T_L$ denote a uniformizing element of $L$ and we let $v_L$
denote the valuation on $L$ normalized so that $v_L(T_L)=1$.  For $s \in \R_{\geq 0}$ there is
an upper numbering ramification group $G_{L/F}^s$ satisfying:

\begin{itemize}
\item When $t>s$ we have $G_{L/F}^t \subset G_{L/F}^s$ .
\item The group $G_{L/F}^0$ is equal to $I_{L/F}$, the inertia group of $G_{L/F}$.
\item The intersection \[\bigcap_{s\geq 0} (G_{L/F})^s = \{0\}\].
\item Let $K$ be a finite extension of $F$ contained in $L$.  Then
$G_{K/F}^s=\frac{G_{L/F}^sG_{L/K}}{G_{L/K}}$.

\end{itemize}

Define the function
\[ \psi_{L/F}(y) = \int_{0}^y [G_{L/F}^0:G_{L/F}^s] ds. \]
Since the function $\psi_{L/F}$ is monotone increasing there
is an inverse function $\phi_{L/F}$.  The \emph{ramification polygon} of $L$ over $K$ is the
 graph of $y=\phi_{L/F}(x)$.  When $L$ is a finite extension of $F$
 there are lower ramification groups satisfying
 $(G_{L/F})_x=(G_{L/F})^{\phi_{L/F}(x)}$.  These lower ramification groups
 may defined explicitly as follows:
 \[ (G_{L/F})_x = \{ g \in G_{L/F} ~ | ~ v_L(g(T_L) - T_L) \geq x\}.\]
 In the case where $[L:F]$ is finite this filtration is finite.  That is,
 for large enough $s$ the groups $G_{L/F}^s$ and $(G_{L/F})_s$ contain
 only the identity element.  We let $\lambda_{L/F}$ (resp. $\mu_{L/K}$) denote the largest number
 such that $(G_{L/F})_s$ (resp. $G_{L/F}^s$) is not trivial.

 \subsection{Some auxillary results on the higher ramification groups}
In the remainder of this section we will prove several auxiliary lemmas
about the higher ramification groups that 
will be used in the proof of Theorem \ref{main local result}.

\begin{lemma} \label{upper to lower} Let $L$ be a 
totally ramified finite extension of $F$
and let $s>0$ satisfy $G_{L/F}^s=0$.
Then for any $g \in G_{L/F}$ we have 
$v(T_L - g(T_L)) < s$.  

\end{lemma}

\begin{proof} Let $(x_0,y_0)$ be the coordinates
for the last break in the ramification polygon of $G_{L/F}$.  
We remark that $x_0$ is the largest number such that
$(G_{L/F})_{x_0}$ is not trivial.
Since $G_{L/F}^s = 0$ we know that $y_0 < s$.  Let $m = |G_{L/F}|$.
At any $x$ we know that $\phi_{L}'(x) \geq \frac{1}{m}$,
provided the derivative exists, 
with equality for $x>x_0$.  This
implies $y_0\geq \frac{x_0}{m}$, which means that
$(G_{L/F})_{y_0 m} = (G_{L/F})_{sm}= 0$.  Thus
$v_{L}(T_L - g(T_L)) < sm$ for all $g$.  
The inertia degree of $L$ over $F$ is $m$, so that
$mv = v_{L}$, which proves the Lemma.

\end{proof}

\begin{lemma} \label{Other ramification Lemma}
Let $L$ be a totally ramified 
finite extension of $F$.  Then
\begin{equation*} \label{other ramification lemma eq}
v(x - g(x)) \geq v(x) + \frac{\lambda_{L/F}-1}{|G_{L/F}|}
\end{equation*} for any $x \in L^\times$ and
$g \in (G_{L/F})_{\lambda_{L/F}}$.
\end{lemma}

\begin{proof}
Let $x,y \in L$ such that $v(x) \neq v(y)$.  If the Lemma 
holds for both $x$ and $y$, then
the the inequality also holds for $x+y$.  Also, if the Lemma
holds for $cx$ when $c \in F$, then the Lemma holds for $x$.
Consider the decomposition of $L$: 
\[\bigoplus_{i=0}^{[L:F]-1} T_L^{i}F.\]  When $x \in T_L^iF$ and
$y \in T_L^j$ with $i\neq j$, we have $v(x)\neq v(y)$, so by
the above remarks it will suffice to prove the statement for powers
of $T_L$.  We have
\begin{equation*}
T_L^i - g(T_L^i) = (T_L - g(T_L))(T_L^{i-1}+...+g(T_L)^{i-1}).
\end{equation*}
Then $v(T_L-g(T_L))=\frac{\lambda_{L/F}}{|G_{L/F}|}$ and each
term in the sum term \[T_L^{i-1}+...+g(T_L)^{i-1}\] has
$T$-adic valuation $v(T_L^i) - \frac{1}{|G_{L/F}|}$, which implies
the lemma.
\end{proof}

\begin{lemma} \label{Bounding upper break in extension}
Let $K$ be a finite extension of $F$ and let $L$ be a finite extension of 
$K$.  Then
\[ \mu_{L/F} \leq \frac{\lambda_{L/F}}{|G_{K/F}|} + \mu_{K/F}.\]
\end{lemma}

\begin{proof}
By definition we have
 $\phi_{K/F}(\lambda_{K/F}) = \mu_{K/F}$.
For any $s> \lambda_{K/F}$ we have $\phi_{K/F}'(s) = \frac{1}{|G_{K/F}|}$
and therefore $\phi_{K/F}(s) = \mu_{K/F} + \frac{s-\lambda_{K/F}}{|G_{K/F}|}$.
A general property of ramification polygons is
that $\phi_{L/F}(s)\leq \phi_{K/F}(s)$ for all $s$.  In particular
if we have $\lambda_{L/F} \geq \lambda_{K/F}$ then
\begin{eqnarray*}
\mu_{L/F} &=& \phi_{L/F}(\lambda_{L/F}) \\
& \leq  & \phi_{K/F}(\lambda_{L/F}) \\
& = & \mu_{K/F} + \frac{\lambda_{L/F}-\lambda_{K/F}}{|G_{K/F}|} \\
& \leq & \frac{\lambda_{L/F}}{|G_{K/F}|} + \mu_{K/F}.
\end{eqnarray*}
If $\lambda_{L/F} < \lambda_{K/F}$ then
\begin{eqnarray*}
\mu_{L/F} & = &  \phi_{L/F}(\lambda_{L/F}) \\
& \leq & \phi_{K/F}(\lambda_{L/F}) \\
& < & \phi_{K/F}(\lambda_{K/F}) \\
& = & \mu_{K/F}.
\end{eqnarray*}
\end{proof}

\begin{lemma} \label{H1 for cyclic extensions} 
Let $K$ be an extension of $F$ and
let $L$ be an extension of $K$
of degree $p$ such that $(G_{L/K})_s=0$.  Then
$H^1(G_{L/K}, \mathcal{O}_{L})$
is killed by $T^{\ceil{\frac{s}{p}}}$.  

\end{lemma}

\begin{proof}
The ramification filtration of $G_{L/K}$
has one break.  By the Hasse-Arf theorem the 
upper numbering of this break is
an integer $n$.  The lower numbering
of the break is $n$ and we have $n<s$.  That is, $v_{L}(T_{L} - g_0(T_{L})) = n$,
where $g_0$ is a generator of $G_{L/K}$.

Consider $N=\sum_{g\in  G_{L/K}} g$ and $r = 1 - g_0$,
viewed as elements of $\Z[G_{L/K}]$.  Both
$N$ and $r$ define $\mathcal{O}_{K}$-linear morphisms from
$\mathcal{O}_{L}$ to itself.  There is an isomorphism 
(see \cite[Section 8]{Atiyah}):
\[ H^1(G_{L/K}, \mathcal{O}_{L}) = \ker (N) / \im(r), \]
so it suffices to consider $\ker (N) / \im(r)$.  
The kernel of $r$ is $\mathcal{O}_K \subset \mathcal{O}_{L}$ and
$\mathcal{O}_{L}$ has $\mathcal{O}_K$-rank $p$, which means
that the rank of $\im(r)$ as an $\mathcal{O}_K$-module is $p-1$.
Similarly, the image of $N$ is contained in $\mathcal{O}_K$,
so that the rank of $\ker (N)$ as an $\mathcal{O}_K$-module is $p-1$.  
Therefore $\ker (N)/\im(r)$
is a torsion $\mathcal{O}_K$-module.  

We have the following decomposition
\[ \mathcal{O}_L = \mathcal{O}_K \oplus T_L 
\mathcal{O}_K \oplus ... T_L^{p-1} \mathcal{O}_K.\]
Then $\im(r)$ will be the $\mathcal{O}_K$-module 
of rank $p-1$ generated by the elements
$e_i=g(T_L^i)-T_L^i$.  We claim that $v_L(e_i) = n + i-1$.  
First write
 \[ e_i = (g(T_L) - T_L)\sum_{k+j=i-1} T_L^k g(T_L)^j.\]
Since $g(T_L)\equiv T_L \mod T_L^2$ we know that each term in
the sum is equivalent modulo $T_L^i$.
In particular we find 
\[ \sum_{k+j=i-1} T_L^k g(T_L)^j \equiv (p-1)T_L^{i-1} \mod T_L^i,\]
whose $T_L$-adic valuation is $i-1$ and therefore $v_L(e_i)=n+i-1$.

It follows that $e_i$ is contained in $T_L^n \mathcal{O}_L$ and not in
$T_L^n T \mathcal{O}_L$.  Now consider
the $p$-dimensional 
$k$-vector space $T_L^n\mathcal{O}_L/T_L^{n}T\mathcal{O}_L$.  Since 
the $e_i$ all have distinct valuations we know that
their images in $T_L^n\mathcal{O}_L/T_L^{n}T\mathcal{O}_L$ are linearly 
independent.  By Nakayama's Lemma
there is $e' \in T_L^n \mathcal{O}_L$ such that
$e_1,...,e_{p-1},e'$ is a $\mathcal{O}_K$-basis of $T_L^n\mathcal{O}_L$.  
Any $x \in \ker (N)$ can be written uniquely as a linear combination
of $e_1,...,e_{p-1}$ over $K$, since $\ker (N)/\im(r)$ is torsion.  
We know that $T^{\ceil{\frac{n}{p}}} x$ is
in $T_L^n\mathcal{O}_L$ so it can be written uniquely as a linear combination
of $e_1,...,e_{p-1},e'$ over $\mathcal{O}_K$.  The coefficient of
$e'$ has to be zero.  This means $T^{\ceil{\frac{n}{p}}} x$ is in $\im(r)$.  It follows
that $\ker (N)/\im(r)$ is killed by $T^{\ceil{\frac{n}{p}}} x$ and since $n<s$ the
Lemma follows.

\end{proof}

\begin{lemma} \label{cohomology Torsion} Let $L$ be a finite
extension of $F$ and let $s$ be an integer with
 $G_{L/F}^s=0$.  Then $H^1(G_{L/F}, \mathcal{O}_L)$
is killed by $T^s$.
\end{lemma}

\begin{proof} Let $K$ be the maximal tamely ramified extension
of $F$ contained in $L$.  The 
restriction-inflation
sequence gives

\[ 0 \to H^1(G_{{K}/{F}}, \mathcal{O}_{K})
\to  H^1(G_{L/{F}}, \mathcal{O}_L) 
\to H^1(G_{L/{K}}, \mathcal{O}_L) \to  H^2(G_{{K}/{F}}, \mathcal{O}_{K}). \]
We know that $\mathcal{O}_K$ is a projective 
$\mathcal{O}_F[G_{{K}/{F}}]$-module, so $H^i(G_{{K}/{F}}, \mathcal{O}_{K})$
is trivial for $i=1,2$.  Therefore, it suffices to prove the theorem for totally
wildly ramified extensions.

 There is a tower of fields 
$L=L_{-1} \supset L_0 \supset ...\supset L_r=F$ such that
$L_i$ is an extension of degree $p$ over $L_{i+1}$.
We will show $H^1(G_{L/{L_{i}}}, \mathcal{O}_L)$
is $T^s$-torsion by inducting on $i$.  When $i=-1$ there is
nothing to show.  For $i\geq 0$, consider the restriction-inflation
sequence

\[ 0 \to H^1(G_{{L_{i-1}}/{L_{i}}}, \mathcal{O}_{L_{i-1}})
\to  H^1(G_{L/{L_i}}, \mathcal{O}_L) 
\to H^1(G_{L/{L_{i-1}}}, \mathcal{O}_L). \]  By
our induction hypothesis $H^1(G_{L/{L_{i-1}}}, \mathcal{O}_L)$
is $T^s$-torsion, so it suffices to show that 
$H^1(G_{{L_{i-1}}/{L_{i}}}, \mathcal{O}_{L_{i-1}})$
is $T^s$-torsion.  Let $k$ be the integer such
that $kp=|G_{L_{i-1}/F}|$.
Since $G_{L_{i-1}/F}^s=0$ we know from Lemma \ref{upper to lower} that
$v_{T_{L_{i-1}}}(T_{L_{i-1}} - g(T_{L_{i-1}})) < k ps$ for
$g \in G_{L_{i-1}/L_{i}}$.  This means $(G_{L_{i-1}/L_{i}})_{k ps}=0$,
so by Lemma \ref{H1 for cyclic extensions} we know that
$T_{L_{i}}^{ks}$ annihilates  
$H^1(G_{L_{i-1}/L_i}, \mathcal{O}_{L_{i-1}})$.  The Lemma
then follows from the fact that $v_T(T_{L_{i}})=\frac{1}{k}$.

\end{proof}

\subsection{The different}

In this subsection we will prove some auxiliary lemmas relating the different
to higher ramification groups.  These Lemmas will be used in Section \ref{section global results}
together with the Riemann-Hurwitz-Hasse formula to prove asymptotic results on genus growth
(Theorem \ref{Main global theorem}).
Recall that for $L$ a finite separable extension of $F$ we define the different:
\[ \delta_{L/F} = \sum_{i=0}^{\infty} (|G_{L/F})_i| - 1). \]
If $f(X)$ is the
minimal polynomial of $T_L$ then $\delta_{L/F} = v_L(f'(T_L))$.

\begin{lemma} \label{the different and the breaks}
Let $L$ be a totally ramified finite extension of $F$.  Then
\[ \frac{\delta_{L/F}}{|G_{L/F}|} = \mu_{L/F} - \frac{\lambda_{L/F}}{|G_{L/F}|}.\]
\end{lemma}
\begin{proof} This is \cite[Proposition IV.4]{Serre}.
\end{proof}

\begin{lemma} \label{lower bound through upper}
With the same notation as Lemma \ref{the different and the breaks}, let $0=s_0<s_1<...<s_n$ be real numbers
with $s_n = \mu_{L/F}$.  Then
\[ \lambda_{L/F} \leq \sum_{i=1}^n (s_{i} - s_{i-1})[G_{L/F}:G_{L/F}^{s_i}].\]
\end{lemma}
\begin{proof}
This follows because $\lambda_{L/F} = \phi_{L/F}(\mu_{L/F})$.  More
precisely:

\begin{eqnarray*}
\lambda_{L/F} &= & \int_{0}^{s_n} [G_{L/F}:G_{L/F}^s]ds \\
& = & \sum_{i=1}^{n} \int_{s_{i-1}}^{s_{i}} [G_{L/F}:G_{L/F}^s]ds \\
& \leq & \sum_{i=1}^n (s_{i} - s_{i-1})[G_{L/F}:G_{L/F}^{s_i}]. 
\end{eqnarray*}

\end{proof}

\begin{corollary} \label{upper break and different bound}
Let $L$ be a finite extension of $F$.  Then
\[ \frac{\mu_{L/F}}{|G_{L/F}^{\mu_{L/F}}|}\leq \frac{\delta_{L/F}}{|G_{L/F}|}.\]
\end{corollary}
\begin{proof}

Combining Lemma \ref{the different and the breaks} and Lemma
\ref{lower bound through upper} with $s_0=0<s_1=\mu_{L/F}$, we obtain

\begin{eqnarray*} \frac{\delta_{L/F}}{|G_{L/F}|} &\geq & 
\mu_{L/F} - \big(\mu_{L/F} - s_0\big)\frac{1}{|G_{L/F}^{\mu_{L/F}}|} \\
& \geq & \mu_{L/F}\bigg(1 - \frac{1}{|G_{L/F}^{\mu_{L/F}}|}\bigg) \\
& \geq & \frac{\mu_{L/F}}{|G_{L/F}^{\mu_{L/F}}|},
\end{eqnarray*}
where the last inequality comes from the fact that $|G_{L/F}^{\mu_{L/F}}|\geq p$.

\end{proof}

\section{Proof of Theorem \ref{main local result}} \label{section Main result}
In this section we will prove Theorem \ref{main local result}.  Let $F$
be a local field of equal characteristic.  In particular $F=k((T))$,
where $k$ is a finite field.  Just as before we let $v$ be the valuation on $F$
normalized so that $v(T)=1$.  Let $V$ be a $d$ dimensional
$\Q_p$-vector space
and let $\rho: G_F \to GL(V)$ be a continuous representation.  
Fix a $G_F$-stable $\Z_p$-lattice $L\subset V$ of rank $d$.
We let $G$ be the image of $\rho$.  The lattice $L$ defines
a $p$-adic Lie filtration on $G$:
\[ G(n) = \ker (G \to GL(L/p^n L)).\]
We let $F_{n}$ be the fixed field of $G(n)$ and let 
$H(n)=G/G(n)$.  Note that $H(n)$ is the Galois group of $F_n$ over $F$.
Let $\mu_n$ (resp. $\lambda_n$) denote
the largest ramification break with the upper (resp. lower) numbering of $H(n)$.

The proof of Theorem \ref{main local result} will be broken up into three smaller
propositions, which fit together as follows:  
Let $\mathbf{e}=(e_1,...,e_d)$ be a basis of $L$ and let $\mathbf{a}=(a_1,...,a_d)$
be an $\mathcal{E}$-basis of 
$(\widetilde{\mathcal{E}}\otimes V)^{G_F}$ (we know that 
$(\widetilde{\mathcal{E}} \otimes V)^{G_F}$ is
an $\mathcal{E}$-vector space of dimension $d$ by Theorem \ref{Fontaine-Tsuzuki}).  
Then we have a period
matrix $A \in M_{d\times d} (\widetilde{\mathcal{E}})$ satisfying

\[ \label{period equation} \bold{a}=A\bold{e}.  \tag{P}\]
We may assume that $A \in M_{d\times d} (\mathcal{O}_{\widetilde{\mathcal{E}}})$ by multiplying our
basis $\mathbf{a}$ by a power of $p$.  
If $\mathbf{a}$ can be chosen so that $A$ lies in 
$M_{d\times d}(\mathcal{O}_{\widetilde{\mathcal{E}}^r})$, then it will follow that
$(\widetilde{\mathcal{E}}^r \otimes V)^{G_F}$ is an $\mathcal{E}^r$-vector space of dimension $d$.  
In Proposition \ref{D^r does the right thing}
we prove that if $\rho \in \Rep^r(G_F)$ then $A$ may be chosen to be in 
$M_{d\times d}(\mathcal{O}_{\widetilde{\mathcal{E}}^r})$.  
Consequently
$D^r(V)$ is a $(\phi,\nabla)$-module over ${\mathcal{E}^r}$ of
dimension $d$.  This shows that all representations in $\Rep^r(G_F)$ are
$\widetilde{\mathcal{E}}^r$-admissible.  Now let $f: V_1 \to V_2$ be 
a morphism in $\Rep^r(G_F)$.  Then we have a corresponding
morphism $f_{\mathcal{E}}: D(V_1) \to D(V_2)$ by Theorem \ref{Fontaine-Tsuzuki}.  
Using Proposition \ref{fully faithful} we know that extension of scalars functor 
$\Mphiet{\mathcal{E}^r} \to \Mphiet{\mathcal{E}}$ is fully faithful.
In particular, we know from Proposition \ref{fully faithful} that
$f_{\mathcal{E}}$ is actually the base change of a morphism $f_{\mathcal{E}^r}: D^r(V_1) \to D^r(V_2)$.  
This means that $D^r$
defines a functor from $\Rep^r(G_F)$ to  $\Mphiet{\mathcal{E}^r}$ that is fully faithful.

To show that $D^r$ is essentially surjective
when restricted to $\Rep^r(G_F)$ we fix a
 $(\phi,\nabla)$-module $M$ over 
${\mathcal{E}^r}$ of dimension $d$.  Choose a basis $\mathbf{a}=(a_1,...,a_d)$
of $M$.  
Proposition \ref{Trivializing log-decay diffeqs} tells us that $M$ trivializes
over ${\widetilde{\mathcal{E}}^r}$.  Thus 
$L_M=(M\otimes_{\mathcal{E}^r} {\widetilde{\mathcal{E}}^r})^{\phi=1}$
is a $\Q_p$-vector space of dimension $d$ with a continuous action $\rho_M$ of $G_F$.  Let 
$\mathbf{e}_M=(e_{M,1},...,e_{M,d})$ be a basis of $L_M$.  The two bases are related
by a period matrix 
\[ \mathbf{a}=A_M\mathbf{e}_M,\]
whose entries are in ${\widetilde{\mathcal{E}}^r}$.  Then Proposition 
\ref{periods over Er have the right ramification growth} tells us that
$\rho_M \in \Rep^r(G_F)$.  It follows that $D^r$ is essentially surjective when restricted
to $\Rep^r(G_F)$.  

\begin{lemma} \label{lifting Galois invariants}
For $n\geq 1$
we have \[ \label{lifting equation}
(\mathcal{O}_{\widetilde{\mathcal{E}}} \otimes_{\Z_p} L)^{G_F} 
\otimes_{\Z_p} \Z/p^n \Z
\cong (\mathcal{O}_{\widetilde{\mathcal{E}}} \otimes_{\Z_p} L/p^n L)^{G_F}. \tag{Q}\]
That is, any period of $L/p^n L$ lifts to a period of $L$.
\end{lemma}

\begin{proof} 
The left side of (\ref{lifting equation}) 
naturally injects into the right side.  Tensoring
with $\mathcal{O}_{\mathcal{E}_{F_n}}$ gives an injection
\[ f: 
((\mathcal{O}_{\widetilde{\mathcal{E}}} \otimes_{\Z_p} L)^{G_F} 
\otimes_{\Z_p}\Z/p^n \Z)
\otimes_{\mathcal{O}_{\mathcal{E}}} \mathcal{O}_{\mathcal{E}_{F_n}}
\hookrightarrow
(\mathcal{O}_{\widetilde{\mathcal{E}}} \otimes_{\Z_p} L/p^n L)^{G_F}
\otimes_{\mathcal{O}_{\mathcal{E}}} \mathcal{O}_{\mathcal{E}_{F_n}}.\]
We will prove that $f$ is also surjective.  Since
 $\mathcal{O}_{\mathcal{E}_{F_n}}$ is faithfully flat over 
 $\mathcal{O}_{\mathcal{E}}$ the isomorphism $f$ descends to
 give $(\ref{lifting equation})$.  There is an inclusion of
 $(\mathcal{O}_{\widetilde{\mathcal{E}}} \otimes_{\Z_p} L/p^n L)^{G_F}$ in
 $(\mathcal{O}_{\widetilde{\mathcal{E}}} \otimes_{\Z_p} L/p^n L)^{G_{F_n}}$.
 Since $(\mathcal{O}_{\widetilde{\mathcal{E}}} \otimes_{\Z_p} L/p^n L)^{G_{F_n}}$
 is an $\mathcal{O}_{\mathcal{E}_{F_n}}$-module, this gives an injective map
 
 \[g:(\mathcal{O}_{\widetilde{\mathcal{E}}} \otimes_{\Z_p} L/p^n L)^{G_F} 
 \otimes_{\mathcal{O}_{\mathcal{E}}} \mathcal{O}_{\mathcal{E}_{F_n}}
 \hookrightarrow 
 (\mathcal{O}_{\widetilde{\mathcal{E}}} \otimes_{\Z_p} L/p^n L)^{G_{F_n}}. \]
In particular $g \circ f$ is injective.  We will prove that $g \circ f$ is
surjective, which will imply that $f$ is surjective.

Let $\bar{e_i}$ be the image
of $e_i$ in $L/p^nL$.  Since $G_{F_n}$ acts trivially on $L/p^nL$, we see
that \[(\mathcal{O}_{\widetilde{\mathcal{E}}} \otimes_{\Z_p} L/p^n L)^{G_{F_n}}\]
is a free $\mathcal{O}_{\mathcal{E}_{F_n}}/p^n$-module
with basis $\bar{e_1},...,\bar{e_d}$.  Similarly, let $\bar{a_i}$ be the image of $a_i$ in 
$(\mathcal{O}_{\widetilde{\mathcal{E}}} \otimes_{\Z_p} L)^{G_F} 
\otimes_{\Z_p}\Z/p^n \Z$.  Consider the reduction modulo $p^n$ of the
period equation
(\ref{period equation}):

\[ \bar{\bold{a}} = \bar{A} \bar{\bold{e}}. \]
Since $\bar{\bold{a}}$ and $\bar{\bold{e}}$
are fixed by $G_{F_n}$ we see that the same is true for
$\bar{A}$.  In particular, we find 
$\bar{A} \in M_{d \times d} ( \mathcal{O}_{\mathcal{E}_{F_n}}/p^n ).$
It follows that each $e_i$ is contained in the image of $g \circ f$,
which means $g \circ f$ is surjective.

\end{proof}

\begin{proposition} \label{D^r does the right thing}
If $\rho \in \Rep^r(G_F)$ then the period matrix $A$ may be taken to have
entries in $\mathcal{O}_{\widetilde{\mathcal{E}}^r}$.
\end{proposition}
\begin{proof} 

Let $M=({\widetilde{\mathcal{E}}}\otimes L)^{G_F}$ be the $(\phi,\nabla)$-module
corresponding to $\rho$.  The period matrix $A$ is unique up to multiplication
by an element of $B \in GL_d(\mathcal{O}_{\mathcal{E}})$,
which corresponds to choosing a different basis of $M$.  Thus we need to prove that there exists $B \in GL_d(\mathcal{O}_{\mathcal{E}})$
such that $BA \in GL_d(\mathcal{O}_{\widetilde{\mathcal{E}}^r})$.

To find $B$ we will find $B_n \in GL_d(\mathcal{O}_{\widetilde{\mathcal{E}}})$
for each $n\geq 1$ satisfying 
\[B_n \equiv B_{n-1} \mod p^{n-1}, \] 
\[\label{okay} B_n A \mod p^n \in GL_d(\mathcal{O}_{\mathcal{E}_{F_{n}}^{r,c}}/p^n\mathcal{O}_{\mathcal{E}_{F_{n}}^{r,c}}), \tag{*}\]
for a well chosen $c>0$.
Then $B=\lim_{n\to\infty} B_n$ will lie in $GL_d(\mathcal{O}_{\widetilde{\mathcal{E}}})$
and \[BA \in GL_d(\mathcal{O}_{\widetilde{\mathcal{E}}^{r,c}}) \subset GL_d(\mathcal{O}_{\widetilde{\mathcal{E}}^r}).\]
Let $c_0>0$ satisfy $G^{p^{nr}c_0}\subset G(n+1)$ for all $n$
and take $c$ large enough to satisfy
\[ c > 
\frac{p^{r}c_0 + |w_1(T)|}{p^r - 1} 
.\]

We proceed inductively.
We may take $B_1$ to be a power of $T$ times the identity matrix.
This is because $B_1A \mod p$ may be regarded as a matrix with
coefficients in $F^{alg}$, so that multiplying the entries by a large
enough power of $T$ will land in $\mathcal{O}_{F^{alg}} = 
\mathcal{O}_{\widetilde{\mathcal{E}}^{r,c}}/p\mathcal{O}_{\widetilde{\mathcal{E}}^{r,c}}$.  
Note that $e_i \mod pL$ is fixed by 
$G_{F_1}$.  Since $B_1A e_i \mod pL$ is also fixed by $G_{F_1}$ it follows
that $B_1A \mod p$ is contained in $\mathcal{O}_{\mathcal{E}_{F_1}^{r,c}}/p\mathcal{O}_{\mathcal{E}_{F_1}^{r,c}}$.

Now assume the existence of $B_{n}$.  We have the following exact sequence:
\[ 0 \to \mathcal{O}_{F_{n+1}} \to \mathcal{E}_{F_{n+1}}^{r,c,n} / p^{n+1} \mathcal{E}_{F_{n+1}}^{r,c,n}
\to \mathcal{O}_{\mathcal{E}_{F_{n+1}}^{r,c}} / p^{n}\mathcal{O}_{\mathcal{E}_{F_{n+1}}^{r,c}}  \to 0. \]
The surjective map is reduction modulo $p^{n}$ and the kernel
is $p^n\mathcal{E}_{F_{n+1}}^{r,c,n} / p^{n+1}\mathcal{E}_{F_{n+1}}^{r,c,n}\cong \mathcal{O}_{F_{n+1}}$
Applying the exact functor $M \to M \otimes_{\Z_p} L$ gives an
exact sequence of $\Z_p[G_F]$-modules.
\[ 0 \to \mathcal{O}_{F_{n+1}} \otimes_{\Z_p} L/pL \to 
\mathcal{E}_{F_{n+1}}^{r,c,n} / p^{n+1} \mathcal{E}_{F_{n+1}}^{r,c,n} \otimes_{\Z_p} L
\to \mathcal{O}_{\mathcal{E}_{F_{n+1}}^{r,c}} / p^n\mathcal{O}_{\mathcal{E}_{F_{n+1}}^{r,c}}  \otimes_{\Z_p} L \to 0. \]
As $G(n+1)$ acts trivially on each term in this sequence we
regard it as a sequence of $\Z_p[H(n+1)]$-modules.  Taking $H(n+1)$
invariants gives the exact sequence

\[ (\mathcal{E}_{F_{n+1}}^{r,c,n} / p^{n+1} \mathcal{E}_{F_{n+1}}^{r,c,n} \otimes_{} L)^{H(n+1)}
\to (\mathcal{O}_{\mathcal{E}_{F_{n+1}}^{r,c}} / p^n \mathcal{O}_{\mathcal{E}_{F_{n+1}}^{r,c}} \otimes_{} L) ^{H(n+1)}
\xrightarrow{\delta} H^1(H(n+1), \mathcal{O}_{F_{n+1}} \otimes_{} L/pL). \]
Let $\epsilon_i = B_{n}A e_i \in \mathcal{O}_{\widetilde{\mathcal{E}}} \otimes L$.
By our inductive hypothesis (\ref{okay}) we see that reducing
$\epsilon_i$ modulo $p^n$ gives  an element 
$\bar{\epsilon}_i \in \mathcal{O}_{\mathcal{E}_{F_{n+1}}^{r,c}} / p^n\mathcal{O}_{\mathcal{E}_{F_{n+1}}^{r,c}}  \otimes_{\Z_p} L$,
which is Galois invariant.  By assumption we know $H(n+1)^{p^{nr}c_0}=0$.  Lemma
\ref{cohomology Torsion} then shows that 
$H^1(H(n+1), \mathcal{O}_{F_{n+1}} \otimes_{\Z_p} L/pL)$ is
annihilated by $T^{p^{nr}c_0}$.  
Thus there exists $\bar{\delta}_i \in (\mathcal{E}_{F_{n+1}}^{r,c,n} / p^{n+1} \mathcal{E}_{F_{n+1}}^{r,c,n}\otimes_{\Z_p} L)^{H(n+1)}$ that reduces to 
$T^{p^{nr}c_0}\bar{\epsilon}_i$ modulo $p^n$.  Let $\delta_i$
be a Galois invariant lifting of $\bar{\delta}_i$ in 
$\widetilde{\mathcal{E}} \otimes L$, which exists by
Lemma \ref{lifting Galois invariants}.  We have the following relation
\[ \frac{1}{T^{p^{nr}c_0}}\delta_i \equiv \epsilon_i \mod p^n.\]
In particular, if $B_{n+1} \in GL_d(\mathcal{O}_\mathcal{E})$ 
is the unique invertible matrix such
that $\frac{1}{T^{p^{nr}c_0}}\boldsymbol{\delta}=B_{n+1}A \bold{e}$,
where $\boldsymbol{\delta}=(\delta_1,...,\delta_d)^{T}$, then 
$B_{n+1}\equiv B_n \mod p^n$.

It remains to show the entries of
$B_{n+1}$ have the desired growth properties.  We have Teichmuller
expansions:
\begin{eqnarray*}
B_{n+1}A &=& \sum_{i=0}^{n}[X_i]p^i + p^{n+1}S_{n+1} \\
B_n A & = & \sum_{i=0}^{n-1} [X_i]p^i + p^{n}S_n,
\end{eqnarray*}
where the $X_i$ are $d \times d$ matricies in $F^{alg}$.  
Since $\delta_i = T^{p^{nr}c_0}B_{n+1}A e_i$ reduces to
an element of 
$\mathcal{E}_{F_{n+1}}^{r,c,n} / p^{n+1} \otimes_{\Z_p} L/p^{n+1}$,
we know that $T^{p^{nr}c_0} B_{n+1}A \in M_{d\times d} 
(\tb{A}^{r,c,n})$.  It follows that 
\begin{eqnarray*}
T^{p^{nr} c_0}\sum_{i=0}^{n}[X_i]p^i & \in & \tb{A}^{r,c,n}. 
\end{eqnarray*}
We also have \[ \sum_{i=0}^{n-1}[X_i]p^i  \in \tb{A}^{r,c,n},\] by our
inductive assumption that $B_nA \mod p^n$ is in $GL_n(\mathcal{O}_{\mathcal{E}^{r,c}/p^{n}})$.
Since multiplication by $T$ preserves $\tb{A}^{r,c,n}$ 
by Lemma \ref{action of $T$} we know that 
\[T^{p^{nr}c_0}\sum_{i=0}^{n-1}[X_i]p^i \in \tb{A}^{r,c,n},\] which
gives $T^{p^{nr}c_0}[X_n]p^n \in \tb{A}^{r,c,n}$.  Write
$T^{p^{nr}c_0}[X_n]p^n = \sum [Y_i]p^i$ so that $Y_i=0$ for $i<n$.
Our definition of $\tb{A}^{r,c,n}$ tells us that 
$v(Y_n) \geq c -p^{(n-1)r}c - |w_1(T)|$.  Therefore
 \[ v(X_n) \geq c - p^{(n-1)r}c - |w_1(T)|-p^{nr}c_0.\]
The inequality $c > \frac{p^r c_0 + |w_1(T)|}{p^r - 1}$ gives
$c> \frac{p^{nr}c_0 + |w_1(T)|}{p^{nr} - p^{(n-1)r}}$.  Rearranging
this inequality we see that
 \[ v(X_n) \geq c - p^{(n-1)r}c - |w_1(T)|-p^{nr}c_0 > c-p^{nr}c,\]  
which shows \[ B_{n+1}A \mod p^{n+1}\in GL_d(\mathcal{O}_{\mathcal{E}_{F_{n+1}}^{r,c}}/p^{n+1}\mathcal{O}_{\mathcal{E}_{F_{n+1}}^{r,c}}).\]

\end{proof}

\begin{proposition} \label{periods over Er have the right ramification growth}

If the period matrix
$A$ satisfying \[\bold{a} = A \bold{e},\] lies in 
$M_{d\times d}( \mathcal{O}_{\widetilde{\mathcal{E}}^r})$, then
$\rho$ lies in $\Rep^r(G_F)$.

\end{proposition} 

\begin{proof} The period matrix has a Teichmuller expansion:
\[A = \sum_{n=0}^\infty [A_n]p^n,\]
with $A_n \in M_{d\times d} (\mathcal{O}_{\widehat{F^{alg}}})$.  By
multiplying $A$ by a power of $ T^{-1}I_d$ we may assume that $v(A_0) = 0$.  The
condition $A \in M_{d\times d}( \mathcal{O}_{\mathcal{E}^r})$ guarantees 
the existence of $c>0$ such that 
\begin{equation}
 \label{period growth inequality} v(A_n)\geq  -p^{(n+1)r}c.   
\end{equation}
We will prove inductively two
things: The first is that
$\mu_n < p^{rn}c_0$
where 

\begin{equation}
\label{ramification bound} c_0> \max 
\bigg( \frac{p^{r}p^{d^2}c}{p^r-1}, \mu_1 \bigg).
\end{equation}  
The second is that $A_{n-1} \in \widehat{F_n^{perf}}$.
The bound on $\mu_n$ will imply $\rho \in \Rep^r(G_F)$.

When $n=1$ we have $\mu_1 < p^r c_0$ by our definition of $c_0$.  To see
that $A_0 \in F_1^{perf}$ 
consider the action of $G_{F_1}$ on $A\bold{e}$ reduced modulo $p$.  
Since $G_{F_1}$ acts trivially on $L/pL$ and it fixes $A\bold{e}$,
we know that $G_{F_1}$ must fix $A \mod p$.  This shows that
$A_0$ is contained in $\widehat{F^{alg}}^{G_{F_1}}=\widehat{F_1^{perf}}$,
which concludes our base case.  
Now assume the result for $n>1$.  
Since $G_{F_{n+1}}$ fixes $A\bold{e}$ and acts trivially
on $L/p^{n+1}L$ we see that $G_{F_{n+1}}$ fixes $A \mod p^{n+1}$, which gives 
$A_n \in \widehat{F^{alg}}^{G_{F_{n+1}}} = \widehat{F_{n+1}^{perf}}$.

It remains to show that $\mu_{n+1}< p^{r(n+1)}c_0$.  Let $g$ be
a nonzero element of $H(n+1)^{\mu_{n+1}}$.  Then $g$ has order $p$
and if we view $H(n+1)$ as a subgroup of $GL_d(\Z_p/p^{n+1}\Z_p)$ through $\rho$ 
we
may represent $g$ by a matrix $I_d + p^nB$ for some 
$B \in M_{d \times d}(\Z_p)$.  The equation
\[ g(A\bold{e}) - A\bold{e} \equiv 0 \mod p^{n+1}, \] thus
becomes 
\[ \big( \sum_{i=0}^n [g(A_i)]p^r \big)\big(I_d + p^nB)\bold{e}
- \big( \sum_{i=0}^n [A_i]p^r \big) \bold{e} \equiv 0 \mod p^{n+1}.\]
Since $g(A_i)=A_i$ for $i<n$ this yields the equation 
\[ [g(A_n)]p^n - [A_n]p^n \equiv B[A_0]p^n \mod p^{n+1}.\]
In particular we have $v(g(A_n) - A_n)= 0$.  
Lemma \ref{Other ramification Lemma} then gives
\[ 0 \geq v(A_n) + \frac{\lambda_{n+1} -1}{|H(n+1)|}.\]
By applying Inequality (\ref{period growth inequality})
we see that $|H(n+1)|p^{(n+1)r}c>\lambda_{n+1}$.  
If we view $G_{F_n}$ and $G_{F_{n+1}}$ as subgroups of $GL_d(\Z_p)$
it is clear that $\frac{|H(n+1)|}{|H(n)|}=|G_{F_n}/G_{F_{n+1}}|\leq p^{d^2}$.  
Then Lemma \ref{Bounding upper break in extension} 
and our inductive hypothesis that $\mu_n \leq p^{rn}c_0$ gives

\begin{eqnarray*}
\mu_{n+1} & \leq &  \frac{\lambda_{n+1}}{|H(n)|} + \mu_{n} \\
& \leq & \frac{|H(n+1)|}{|H(n)|} p^{(n+1)r}c + \mu_n \\
& \leq & p^{(n+1)r}cp^{d^2} + p^{rn}c_0.
\end{eqnarray*}
However the inequality (\ref{ramification bound}) implies
\[p^{(n+1)r}cp^{d^2} + p^{nr}c_0 < p^{(n+1)r}c_0,\]
from which we see $\mu_{n+1}< p^{(n+1)r}c_0.$

\end{proof}

\begin{proposition} \label{Trivializing log-decay diffeqs}
Let $M$ be a $(\phi,\nabla)$-module over $\mathcal{E}^r$.
Then $M$ trivializes after base changing to $\widetilde{\mathcal{E}^r}$.
\end{proposition}

\begin{proof} Let $a_1,...,a_d$ be a basis of $M$ and let $C$ be
the matrix of $\phi$ with respect to this basis.  
By Fontaine's theory of $\phi$-modules we know that 
\[ M\otimes_{\mathcal{E}^r} \widetilde{\mathcal{E}} \cong
\oplus_{i=1}^d e_i \widetilde{\mathcal{E}}.  \]
The right hand side is the $(\phi,\nabla)$-module given by
 $\phi(e_i)=e_i$ and $\nabla(e_i)=0$.  
Let $A \in GL_d(\widetilde{\mathcal{E}})$ be the period matrix
satisfying $\bold{a}=A\bold{e}$, where $\bold{a}=(a_1,...,a_d)$ and
$\bold{e}=(e_1,...,e_d)$.  It is enough to show that 
$A$ has entries in $\widetilde{\mathcal{E}}^r$.  We have $\phi(\bold{a})=C\bold{a}$ and
$\phi(\bold{a})=\phi(A\bold{e})=\sigma(A)A^{-1}\bold{a}$.  This gives
$CA=\sigma(A)$.  Since $w_n(\sigma(A))=pw_n(A)$ we have
\begin{equation*}
pw_n(A) = w_n(CA) \\
\geq \min_{i+j\leq n} w_i(C) + w_j(A) \\
\geq w_n(C) + w_n(A).
\end{equation*}
Therefore $w_n(A)\geq \frac{w_n(C)}{p-1}$.  Since $C$ is contained in 
$M_{d\times d}(\tb{A}^r)$,
we see that the same is true of $A$.  Therefore the entries of
$A$ are contained in $\tb{A}^r\cap \widetilde{\mathcal{E}}$.

\end{proof}

\begin{proposition} \label{fully faithful}
 Let $M$ and $N$ be $(\phi,\nabla)$-modules over $\mathcal{E}^r$ and
 let $f_\mathcal{E}:M \otimes_{\mathcal{E}^r} \mathcal{E} \to N \otimes_{\mathcal{E}^r}\mathcal{E}$ be a morphism of
 $(\phi,\nabla)$-modules over $\mathcal{E}$.  Then $f_\mathcal{E}$ descends to a morphism
 $f_{\mathcal{E}^r}: M \to N$ defined over $\mathcal{E}^r$.

\end{proposition}

\begin{proof}
 Let $e_1,...,e_m$ be a basis of $M$ and let $f_1,...,f_n$ be a basis of $N$.  Let
 $A$ (resp. $B$) be the matrix for the Frobenius of $M$ (resp. $N$) with respect
 to $e_1,...,e_m$ (resp. $f_1,...,f_n$).  Let $S \in M_{m \times n}(\mathcal{E})$ be the matrix of $f$
 with respect to these bases.  We have to show that the entries of $S$ lie in
 $\mathcal{E}^r$.  The compatibility of $f$ with Frobenius gives
 \[ B^{-1}SA = S^\sigma,\]
 where $S^\sigma$ denotes $\sigma$ applied to each entry of $S$.  Let $d$ be large enough
 so that \[w_n(B^{-1}),w_n(A) \geq -p^{rn}d.\]  Then we have
 \begin{eqnarray*}
  w_n(S^\sigma)  & \geq & \min_{i+j+k=n} \{ w_i(B^{-1}) + w_j(S) + w_k(A) \} \\
  & \geq & -2p^{rn}d + w_n(S).
 \end{eqnarray*}
Since $w_n(S^\sigma)= pw_n(S)$ we see that $w_n(S) \geq \frac{-2p^{rn}d}{p-1}$, which
means that $S$ has entries in $\tb{B}^r$.  Since the entries of $S$ are also in $\mathcal{E}$ we know from
Lemma \ref{naive log vs witt log} that $S$ has entries in $\mathcal{E}^r$.

\end{proof}

\begin{remark} The proof of Proposition \ref{fully faithful} can be adapted to show that
 the base change functor from $\Mphiet{\mathcal{E}^\dagger}$ to $\Mphiet{\mathcal{E}}$ is
 fully faithful.  This was first proved by Tsuzuki (see \cite{Tsuzuki3}).
\end{remark}

\section{Global $r$-log-decay $F$-isocrystals and genus growth}
\label{section global results}

To define overconvergent $F$-isocrystals, Berthelot defines a sheaf of rings $j^\dagger \mathcal{O}_{\mathcal{X}^{an}}$
on $\mathcal{X}^{an}$ (\cite[Chapter 2]{Berthelot1}).  The sheaf $j^\dagger \mathcal{O}_{\mathcal{X}^{an}}$ agrees with the structure sheaf
$\mathcal{O}_{\mathcal{Y}^{an}}$ when restricted to $\mathcal{Y}^{an}$.  On any strict neighborhood 
$\mathcal{Y}^{an} \subset V$ the sections $\Gamma(V,j^\dagger \mathcal{O}_{\mathcal{X}^{an}})$ consist
of functions on $\mathcal{Y}^{an}$ that overconverge into each disc $]x[ \subset ]D[$.  
Additionally, for any open neighborhood $V$ contained in $]D[$ there are no sections
of $j^\dagger \mathcal{O}_{\mathcal{X}^{an}}$.  An $F$-isocrystal
on $\mathcal{O}_{\mathcal{Y}^{an}}$ is then overconvergent if it extends to an $F$-isocrystal on 
$j^\dagger \mathcal{O}_{\mathcal{X}^{an}}$.  Following Crew (\cite[Section 4]{Crew3}) one may \emph{localize} at each $x \in D$
to obtain an $F$-isocrystal defined over $\mathcal{E}^{\dagger}$.
In this section we will define a sheaf of rings $\mathcal{O}_{\mathcal{Y}^{an}}^r$ on $\mathcal{X}^{an}$ that will capture
the property of $r$-log-decay at each $x \in D$.  We will prove that an $F$-isocrystal over $\mathcal{O}_{\mathcal{Y}^{an}}$
extends to an $F$-isocrystal over $\mathcal{O}_{\mathcal{Y}^{an}}^r$ if and only if at each point $x \in D$ the localized
$F$-isocrystal can be defined over $\mathcal{E}^r$.  Finally, we will use the Riemann-Hurwitz-Hasse formula
to relate the $r$-log-decay property to genus growth bounds for pro-$p$ towers of curves over $X$.  This is all summarized
in Theorem \ref{Main global theorem}.

\subsection{The $r$-log-decay sheaves $\mathcal{O}_{\mathcal{Y}^{an}}^r$} 
\label{log decay sheaf description}
Let $x \in D$ and let $t_x$ be a rational function on $X$ with
a simple zero at $x$ (i.e. $t_x$ is a local parameter of $x$).
Then $t_x$ lifts to a function $T_x$, which is defined on an open
subset of $\mathcal{X}$, and $\hat{\mathcal{O}}_{\mathcal{X},x}$ is isomorphic to $\mathcal{O}_K[[T_x]]$.  
Let $\underline{\mathcal{E}_{T_x}}|_{]x[}$ be the constant sheaf of the ring
$\mathcal{E}_{T_x}$ on the tube $]x[$ over $x$.
Consider the immersions:
\begin{eqnarray*}
 i_x &  :  & ]x[  ~~\hookrightarrow  \mathcal{X}^{an} \\
 j^{an} &  :  & \mathcal{Y}^{an}  \hookrightarrow  \mathcal{X}^{an}.
\end{eqnarray*}
We will describe a map \[ f_x: j^{an}_* \mathcal{O}_{\mathcal{Y}^{an}} \to (i_x)_*(\underline{\mathcal{E}_{T_x}}|_{]x[}). \]
Let $\mathcal{V}^{an}$ be a connected affinoid subspace of $\mathcal{X}^{an}$.  If $\mathcal{V}^{an}$ is
contained entirely in $]x[$ (resp. $\mathcal{Y}^{an}$) then the target (resp. the source)
is empty and there is nothing to describe.  When $\mathcal{V}^{an}$ has a nontrivial intersection with 
both $\mathcal{V}^{an}$ and $]x[$, the map $f_x$ will roughly send a section of 
$\Gamma(\mathcal{V}^{an},j^{an}_* \mathcal{O}_{\mathcal{Y}^{an}})=\Gamma(\mathcal{V}^{an}\cap \mathcal{Y}^{an}, \mathcal{O}_{\mathcal{Y}^{an}})$
to its $T_x$-adic expansions.  Let us make this more precise.  
By Lemma 4.3 in \cite{Crew3}, there exists $r$ between $0$ and $1$
such that $\mathcal{V}^{an}$ contains the annulus defined by $r\leq |T_x|_p < 1$.  Let
$\mathcal{R}_{[r,1)}$ denote the ring of analytic functions on 
$r\leq |T_x|_p < 1$.  In particular we have a restriction map 
$\Gamma(\mathcal{V}^{an}, \mathcal{O}_{\mathcal{X}^{an}}) \to \mathcal{R}_{[r,1)}$.
The ring $\mathcal{R}_{[r,1)}$ consists of Laurent series 
\[\sum_{n=-\infty}^{\infty} a_nT_x^n \]
with coefficients in $K$ such that $\sup |a_n|_pr_0^n  <\infty$ for all $r\leq r_0 <1$.  
The subring $\mathcal{R}^b_{[r,1)} \subset \mathcal{R}_{[r,1)}$
of bounded analytic functions consists of those Laurent series in 
$\mathcal{R}_{[r,1)}$ where the $a_n$ are bounded.  By the maximum modulus
principle the functions in $\Gamma(\mathcal{V}^{an}, \mathcal{O}_{\mathcal{X}^{an}})$
are bounded, which means the image of the restriction map lies in $\mathcal{R}^b_{[r,1)}$.
Putting this together gives

\begin{center}
\begin{tikzcd}

\Gamma(\mathcal{V}^{an}, \mathcal{O}_{\mathcal{X}^{an}}) \arrow{r} \arrow{d}
&
\Gamma(\mathcal{V}^{an},j^{an}_* \mathcal{O}_{\mathcal{Y}^{an}})=
\Gamma(\mathcal{V}^{an}, \mathcal{O}_{\mathcal{X}^{an}})\ang{T_x^{-1}}
\arrow{d} \\
\mathcal{R}^b_{[r,1)} \arrow{r} & \mathcal{E}_{T_x}.

\end{tikzcd}
\end{center}

This is essentially the localization process that Crew describes in \cite[Section 4]{Crew3},
which allows us to pass from global $F$-isocrystals on $\sheafrig{Y}$ to local $F$-isocrystals
on $\mathcal{E}$.
We would like to \emph{pick out} the sections of 
$\Gamma(\mathcal{V}^{an},j^{an}_* \mathcal{O}_{\mathcal{Y}^{an}})$
that are sent to $\mathcal{E}_{T_x}^r$ by $f_x$.  To do this we utilize 
the constant sheaf $\underline{\mathcal{E}_{T_x}^r}|_{]x[}$ of the
ring $\mathcal{E}_{T_x}^r$ on $]x[$.  Consider the pullback 
$\mathcal{O}_{\mathcal{Y}^{an},x}^r$ that makes the following diagram 
Cartesian

\begin{center}
\begin{tikzcd}
 \mathcal{O}_{\mathcal{Y}^{an},x}^r \arrow{r} \arrow{d} & 
 j^{an}_* \mathcal{O}_{\mathcal{Y}^{an}} \arrow{d}{f_x} \\
  (i_x)_*(\underline{\mathcal{E}_{T_x}^r}|_{]x[}) \arrow{r} & 
  (i_x)_*(\underline{\mathcal{E}_{T_x}}|_{]x[})
 \end{tikzcd}
  
 \end{center}
  Taking the sections of this diagram on $\mathcal{V}^{an}$ gives
  the Cartesian diagram
\begin{center}
\begin{tikzcd}

\Gamma(\mathcal{V}^{an}, \mathcal{O}^r_{\mathcal{Y}^{an},x}) \arrow{r} \arrow{d}
&
\Gamma(\mathcal{V}^{an},j^{an}_* \mathcal{O}_{\mathcal{Y}^{an}})
\arrow{d} \\
\mathcal{E}_{T_x}^r \arrow{r} & \mathcal{E}_{T_x}.

\end{tikzcd}
\end{center}
In particular, we see that $\Gamma(\mathcal{V}^{an}, \mathcal{O}^r_{\mathcal{Y}^{an},x})$
consists of the functions on $\mathcal{V}^{an} \cap \mathcal{Y}^{an}$
that have $r$-log-decay around $x$.  We remark that by Lemma \ref{log extensions} the construction of 
$\mathcal{O}^r_{\mathcal{Y}^{an},x}$ does depend on the parameter $T_x$.

Finally, we define $\mathcal{O}_{\mathcal{Y}^{an}}^r$ by following this construction
simultaneously for each point $x \in D$.  Consider the sheaves on $\mathcal{X}^{an}$
\begin{eqnarray*}
\underline{\mathcal{E}_D} &:=& \bigoplus_{x\in D} (i_x)_*(\underline{\mathcal{E}_{T_x}}|_{]x[})\\
\underline{\mathcal{E}^r_D} &:=& \bigoplus_{x\in D} (i_x)_*(\underline{\mathcal{E}^r_{T_x}}|_{]x[})
\end{eqnarray*}
and the map of sheaves 
\begin{eqnarray*}
f_D &:=& \bigoplus_{x\in D} f_x:  j^{an}_* \mathcal{O}_{\mathcal{Y}^{an}} \to \underline{\mathcal{E}_D}.
\end{eqnarray*}
Then $\mathcal{O}_{\mathcal{Y}^{an}}^r$ is taken to be the sheaf that completes
the Cartesian diagram:

\begin{center}
\begin{tikzcd}

 \mathcal{O}_{\mathcal{Y}^{an}}^r \arrow{r} \arrow{d}
&
j^{an}_* \mathcal{O}_{\mathcal{Y}^{an}}
\arrow{d}{f_D} \\
\underline{\mathcal{E}^r_D} \arrow{r} & \underline{\mathcal{E}_D}.

\end{tikzcd}
\end{center}

\begin{remark} If we were to replace $\mathcal{E}^r_{T_x}$ with $\mathcal{E}^\dagger_{T_x}$
we would recover Berthelot's overconvergent sheaf.  However, Berthelot's construction avoids
choosing any local parameters and extends to more general rigid varieties.
It would be interesting to have a less ad-hoc construction of $\mathcal{O}_{\mathcal{Y}^{an}}^r$
that does not involve choosing any parameters and that works for higher dimensional varieties.

\end{remark}
\subsection{Global results and genus growth}
To state our main theorem we must explain how a $p$-adic representation of $\pi_1(Y)$ gives
rise to a pro-$p$ tower of curves.  
Let $\rho: \pi_1(Y) \to GL_d(\Z_p)$ be a continuous $p$-adic representation. 
We define
$G=\im(\rho)$ and let \[G=G(0)\supset G(1) \supset G(2) \supset...\] be
the $p$-adic Lie filtration given by $G(n) = \ker(G \to GL_d(\Z_p/p^n\Z_p))$.
This gives rise to a pro-$p$ tower of curves 
\[X=X_0 \leftarrow X_1 \leftarrow X_2 \leftarrow...\]
that is \'etale outside of $D$.  
We define $g_n$ to be the genus of $X_n$ and we define
$d_n$ to be the degree of $X_n$ over $X$ (i.e. $d_n=|G/G(n)|$).  Our global
interpretation of the $r$-log-decay property has to do with the asymptotic
growth of $\frac{g_n}{d_n}$.

\begin{theorem} \label{Main global theorem} The following categories are equivalent

\begin{enumerate}
\item The category of continuous $p$-adic representations  
\[\rho:\pi_1(Y) \to GL_n(\Q_p) \]
such for each $x \in D$ with decomposition group $G_{F_x}$ the restriction
$\rho|_{G_{F_x}}$ lies in $\Rep^r(G_{F_x})$.

\item The category of continuous $p$-adic representations  
\[\rho:\pi_1(Y) \to GL_n(\Q_p) \]
such that any pro-$p$ tower of curves obtained from a $\pi_1(Y)$-stable lattice as above
has the property that $\frac{g_n}{d_n}$ grows $O(p^{n(r+1)})$.

\item The category of unit-root convergent $F$-isocrystals $\mathcal{M}$ on $\mathcal{Y}^{an}$
such that for each $x \in D$ the localized $F$-isocrystal $M_x$ over $\mathcal{E}_{T_x}$
descends to an $F$-isocrystal over $\mathcal{E}_{T_x}^r$.  The morphisms are just morphisms
of $F$-isocrystals.

\item The category of unit-root convergent $F$-isocrystals $M$ on $\mathcal{Y}^{an}$
that extend to $\mathcal{O}_{\mathcal{Y}^{an}}^r$.
\end{enumerate}
 
\end{theorem}

\begin{proof} The equivalence of (1) and (3) follows from applying
Theorem \ref{main local result} to each point $x \in D$.  To prove the equivalence of
(1)
and (2) we will use the Riemann-Hurwitz-Hasse formula.  Fix a $\Z_p$-lattice
that is invariant under $\pi_1(Y)$ and consider the setup described 
above with the same notation.
For $x \in D$
let $x_n \in X_n$ be a compatible system of points over $x$.  We define
$F_{x_n}$ to be the fraction field of $\hat{\mathcal{O}}_{X_n, x_n}$, which
gives a tower of local fields \[ F_{x_0}\subset F_{x_1} \subset F_{x_2} \subset ...\]
whose union we denote $F_{x_\infty}$.
The Riemann-Hurtwitz-Hasse formula (see \cite[Theorem 6.1]{Oort}) states
\[ g_n - 2 = d_n (g_0-2) + \sum_{x' \in X_n} \delta_{x'},\]
where $\delta_{x'}$ denotes the different of the map $X_n \to X$
at $x'$.  There are $\frac{d_n}{|F_{x_n}:F_{x}|}$ points of $X_n$ above $x \in D$
and each point gives the same different.  This gives
\begin{equation*} \label{genus different relation}
\frac{g_n}{d_n} = g_0 - 2 + 
\sum_{x \in D} \frac{\delta_{x_n}}{|F_{x_n}:F_{x}|} + \frac{2}{d_n}. \tag{S}
\end{equation*}
If we have the local ramification bounds in the hypothesis of (1) 
we know from Lemma \ref{the different and the breaks} that
$\frac{\delta_{x_n}}{|F_{x_n}:F_{x}|}$ is $O(p^{rn})$ for each
$x \in D$.  It follows from (S) that $\frac{g_n}{d_n}$ is $O(p^{nr})$.  This proves
that (1) implies (2).

Conversely, assume that $\frac{g_n}{d_n}$ is $O(p^{rn})$.  For $x \in D$ we know by (S) 
that $\frac{\delta_{x_n}}{|F_{x_n}:F_{x}|}$ grows $O(p^{rn})$.  Let $c>0$ such that
$\frac{\delta_{x_n}}{|F_{x_n}:F_{x}|}\leq cp^{rn}$ for all $n$.  Denote by $\mu_{x_n}$
the largest upper numbering ramification break of $F_{x_n}$ over $F_{x}$.
Equivalently we may define $\mu_{x_n}$ as the smallest number such that
$G_{F_{x_\infty}/F_{x}}^{\mu_{x_n}+\epsilon} \subset G_{F_{x_\infty}/F_{x_n}}$ for
any $\epsilon>0$.
We need to prove that $\mu_{x_n}$ is $O(p^{rn})$.  By 
Corollary \ref{upper break and different bound} we know that

\begin{eqnarray*}
cp^{rn} &\geq&  \frac{\delta_{x_n}}{|F_{x_n}:F_{x}|} \\
& \geq  & \frac{\mu_{x_n}}{|G_{F_{x_n}/F_{x}}^{\mu_{x_n}}|}.
\end{eqnarray*}
View $G_{F_{x_n}/F_{x}}$ as a subgroup of $GL_d(\Z_p/p^{n}\Z_p)$.  The elements
of $G_{F_{x_n}/F_{x}}^{\mu_{x_n}}$ all have order $p$, so that 
$g \in G_{F_{x_n}/F_{x}}^{\mu_{x_n}}$ corresponds to a matrix
that reduces to the identity modulo $p^{n-1}$.  This gives us
the bound \[ |G_{F_{x_n}/F_{x}}^{\mu_{x_n}}|\leq p^{d^2},\] from which we see
\[cp^{d^2}p^{rn} \geq \mu_{x_n}.\]

It remains to show the equivalence between (3) and (4)..Let $\mathcal{M}^r$ be an $F$-isocrystal of $\sheafrig{Y}^r$-modules.  Then
we obtain an $F$-isocrystal $\mathcal{M}$ of $\sheafrig{Y}$-modules by restricting
$\mathcal{M}^r$ to $\mathcal{Y}^{an}$.  By our construction of $\sheafrig{Y}^r$ we 
know that $\mathcal{M}$ will have the desired local properties.  This gives a
functor from (4) to (3).  To see that this functor is fully faithful,
we may follow Crew's argument in \cite[4.6-4.10]{Crew3} and then
apply Proposition \ref{fully faithful} (note that this is the same approach that
Tsuzuki takes to proving the fully faithfullness of overconvergent
$F$-isocrystals in \cite[Theorem 5.1.1]{Tsuzuki3}).  To prove
that this functor is essentially surjective is the bulk of the work.  
Informally, what we need to
prove that if a global $F$-isocrystal locally descends to an $F$-isocrystal
with $r$-log-decay, then we can descend this $F$-isocrystal to have
$r$-log-decay globally.  
Let $\mathcal{M}$ be an $F$-isocrystal on $\mathcal{Y}^{an}$ satisfying the conditions
of (3).  We will carefully choose a covering of $\mathcal{X}^{an}$ by tubes
and prove that $M$ globally has $r$-log-decay when restricted to each tube
in our covering.  These tubes will be chosen so that gluing the $r$-log-decay
$F$-isocrystals is trivial.

For $x \in D$ we may find a rational function $\overline{f_x}$ on $X$ such that
$\overline{f_x}$ has a simple zero at $x$ and $\overline{f_x}$ has a pole at each 
$x' \in D-\{x\}$.  By allowing these poles to have high enough
orders we may use Riemann-Roch to ensure that $\overline{f_x}$ and $\overline{f_{x'}}$ 
have no common zeros whenever $x\neq x'$.  Let $V_x$ be the largest
Zariski open subset of $X$ on which $\overline{f_x}$ is defined and so that
the only zero of $\overline{f_x}$ on $V_x$ is $x$ (i.e. we take domain of definition
of $\overline{f_x}$ and then remove all the zeros of $\overline{f_x}$ except for $x$).  Since
the $\overline{f_x}$ have no common zeros we have 
\[X=\bigcup_{x\in D} V_x.\]  Also, by the way we chose the poles of $\overline{f_x}$
we have \[V_x\bigcap V_{x'} \subset Y\] whenever $x\neq x'$.

Let $\mathcal{V}^{an}_x$ be $]V_x[$ and let $\mathcal{W}^{an}_x$ be
$\mathcal{V}^{an}_x-]x[$.  Note that $\mathcal{W}^{an}_x\subset \mathcal{Y}^{an}$
because $V_x - \{x\}$ is contained in $Y$.
We will let $\sheafrig[x]{V}^r$ be the restriction
of $\sheafrig{X}^r$ to $\mathcal{V}^{an}_x$.
We will prove that the $F$-isocrystal
$\mathcal{M}|_{\mathcal{W}^{an}_x}$ extends to an $F$-isocrystal 
$\mathcal{M}_x^r$ of $\sheafrig[x]{V}^r$-modules.  That is, when we restrict $\mathcal{M}_x^r$ to
$\mathcal{W}^{an}_x$ we obtain an $F$-isocrystal of $\sheafrig[x]{W}$-modules 
isomorphic to $\mathcal{M}|_{\mathcal{W}^{an}_x}$ (recall that since $\mathcal{W}^{an}_x \subset
\mathcal{Y}^{an}$ the restriction of $\sheafrig[x]{V}^r$ to $\mathcal{W}^{an}_x$ is
$\sheafrig[x]{W}$).  We immediately see that $\mathcal{M}_x^r$ and $\mathcal{M}_{x'}^r$ are isomorphic when restricted
to $\mathcal{W}_x^{an}\cap \mathcal{W}_{x'}^{an}$, which means we can patch the $\mathcal{M}_x^r$
together an obtain a sheaf $\mathcal{M}^r$ of $\sheafrig{Y}^r$-modules that restricts to
$\mathcal{M}$ on $\mathcal{Y}^{an}$.

Let $\mathcal{V}_x$ be the formal open subscheme of
$\mathcal{X}$ corresponding to 
$V_x \subset X$ and let $\mathcal{W}_x \subset \mathcal{V}_x$ be the formal open subscheme
corresponding to $V_x - \{x\}$.  By \cite[Lemma 2.5.1]{Crew3} we may find a 
small enough affine formal neighborhood $\mathcal{U}_x=\Spf(A_x)$ of $x$ contained in
$V_x$ such that there is a free coherent sheaf on 
$\Spf(A_x)-\{x\} = \Spf(A_x\ang{f_x^{-1}})$ whose rigid analytification is 
$\mathcal{M}$ restricted to $\Sp(A_x\ang{f_x^{-1}}\otimes \Q)=\mathcal{U}_x^{an}-]x[$.  In 
particular, if we let $\mathcal{T}_x=\Spf(A_x)-\{x\}$ then $\mathcal{M}|_{\mathcal{T}_x^{an}}$ is free.  
Note that
$\mathcal{U}_x^{an}$ is equal to the tube $]U_x[$ of an open set $U_x$ containing
$x$, which means that $\mathcal{V}_x$ is covered by $\mathcal{U}_x^{an}$
and $\mathcal{W}_x^{an}$.  We will prove that $\mathcal{M}|_{\mathcal{T}_x^{an}}$ extends
to a sheaf $\mathcal{M}_{\mathcal{U}_x}^r$ of $\sheafrig[x]{U}^r$-modules, 
utilizing the fact that $\mathcal{M}|_{\mathcal{T}_x^{an}}$ is free.  Then 
$\mathcal{M}_{\mathcal{U}_x^{an}}^r$ and $\mathcal{M}|_{\mathcal{W}_x^{an}}$ are
isomorphic on $\mathcal{T}_x \cap \mathcal{W}_x$ and we may glue them together
to get a sheaf of $\sheafrig[x]{V}^r$-modules $\mathcal{M}_x^r$ that extends $\mathcal{M}|_{\mathcal{W}_x^{an}}$.

We are now reduced to the case where $\mathcal{M}|_{\mathcal{T}_x^{an}}$ is free.  
In particular, if we let $R_x$ be $\Gamma(\mathcal{T}_x^{an}, \sheafrig[x]{T})$ then $\mathcal{M}|_{\mathcal{T}_x^{an}}$ corresponds to a free $R_x$-module
$M_x$.
Let $T_x$ be a lifting of
$\overline{f_x}$ in $\Gamma(\mathcal{W}_x, \sheafformal[x]{W})$.  The
function $T_x$ then defines a map $\mathcal{W}_x^{an}$ to $\mathbb{G}_m^{an}$,
where $\mathbb{G}_m^{an}$ denotes the rigid analytic multiplicative group
$\Sp(K\ang{T_x,T_x^{-1}})$.  Since $\mathcal{T}_x^{an}$ is contained in 
$\mathcal{W}_x^{an}$ we have a commutative diagram: 

\begin{center}
\begin{tikzcd}

K\ang{T_x,T_x^{-1}} \arrow{r} \arrow{d} & \mathcal{E}_{T_x} \arrow{d} \\
R_x \arrow{r} & \mathcal{E}_{T_x},

\end{tikzcd}
\end{center}

where the top horizontal map is just the inclusion and the bottom horizontal map
is the "expand in terms of $T_x$" map described in \ref{log decay sheaf description}.  
We will show that there exists an $F$-isocrystal defined over $R_x\cap \mathcal{E}_{T_x}^r$
that becomes isomorphic to $M_x$ after base-changing to $R_x$.

Our condition in (3) means that there is an $F$-isocrystal $M_x^r$ defined over
$\mathcal{E}_{T_x}^r$ such that 
\[ M_x \otimes_{R_x}\mathcal{E}_{T_x} \cong 
M_x^r \otimes_{\mathcal{E}_{T_x}^r} \mathcal{E}_{T_x}.\]
Let $\mathbf{e}=(e_1,...,e_n)$ be a basis of $M_x$ and let $\mathbf{f}=(f_1,...,f_n)$
be a basis of $M_x^r$.  Consider the matrix $L\in M_{n\times n}(\mathcal{E}_{T_x})$
satisfying $L\mathbf{e}=\mathbf{f}$.  
By Lemma \ref{Transition matrix converges on Gm} there exists 
$B \in M_{n\times n}(\mathcal{E}_{T_x}^r)$ such that $BL$ has entries in
$K\ang{T_x,T_x^{-1}}$.  In particular, we can view $BL$ as having entries
in $R_x$.  Since the Frobenius and connection matrices 
in terms of the basis $\mathbf{e}$ are
defined over $R_x$, the same is true for the basis $BL\mathbf{e}=\mathbf{g}$.
Similarly the Frobenius and connection matrices in terms of $\mathbf{f}$ are
defined over $\mathcal{E}_{T_x}^r$, the same is true for $B\mathbf{f}=\mathbf{g}$.
It follows that the $F$-isocrystal structure in terms of the basis $\mathbf{g}$
is defined over $\mathcal{E}_{T_x}^r\cap R_x$.  The two bases $\mathbf{e}$ and
$\mathbf{g}$ are related by a matrix in $R_x$, so we that the corresponding
$F$-isocrystals are isomorphic over $R_x$.

\end{proof}

\begin{lemma} \label{Transition matrix converges on Gm}
Let $L \in GL_n(\mathcal{E})$.  There exists $B \in GL_n(\mathcal{O}_K[[T]][T^{-1},p^{-1}])$
such that $BL$ has entries in $K\ang{T,T^{-1}}$.
\end{lemma}
\begin{proof}
After multiplying $L$ by powers of $p$ and $T$ we may assume that 
$L \in GL_n(\mathcal{O}_{\mathcal{E}})$ and that the reduction of $L$ modulo $p$
lies in $k[[T]]$.  In particular, we may write
\[L(T) = \sum_{n=-\infty}^\infty l_nT^n,\] where $l_n \in M_{d\times d}(\mathcal{O}_K)$.
We let $v_p(l_n)$ denote the infimum of the $p$-adic valuation of the entries
in $l_n$.  Then $L$ having entries in $\mathcal{O}_{\mathcal{E}}$ is equivalent
to $v_p(l_n)\geq 0$ for all $n$ and $v_p(l_n)\to \infty$ as $n\to -\infty$. 
Since \[\mathcal{O}_K\ang{T,T^{-1}} \cong \varprojlim_{n\to\infty} \mathcal{O}_K/p^n\mathcal{O}_K [T,T^{-1}],\]
it will suffice to find $B \in GL_n(\mathcal{O}_K[[T]][T^{-1},p^{-1}])$ such that $BL$ reduces
to a polynomial in $T$ and $T^{-1}$ modulo every power of $p$.  We do this by successive
$p$-adic approximation.

Let $k_0(T) \in GL_n(\mathcal{O}_K[[T]])$ be a matrix of power series that reduces
to $L(T)^{-1}$ modulo $p$.  This gives
\[ k_0(T)L(T) \equiv 1 + T^{c_1}r_1(T)p \mod p^2,\]
where we may take $r_1(T)$ to be a matrix with entries in $\mathcal{O}_K[[T]]$ 
that is invertible in $\mathcal{O}_K[[T]]$ (i.e. $T$ does not divide $r_1$ and the 
constant term is a $p$-adic unit).  We may find a matrix
$s_1(T)$ with entries in  $\mathcal{O}_K[[T]]$ so that 
\[a_1(T)=s_1(T) + T^{c_1}r_1(T)\]
has no terms of positive degree.  Then $k_1(T) = 1 + s_1(T)p$
satisfies the congruence
\[k_1(T)k_0(T)L(T) \equiv 1 + a_1(T)p \mod p^2, \]
which lies in $\mathcal{O}_K/p^2 \mathcal{O}_K[T,T^{-1}]$.  Following this pattern, we let 
$r_2(T)$ be a matrix with entries in $\mathcal{O}_K[[T]]$ such that
\[k_1(T)k_0(T)L(T) \equiv 1 + a_1(T)p + T^{c_2}r_2(T)p^2 \mod p^3. \]
Then we find $s_2(T)$ with entries in $\mathcal{O}_K[[T]]$ so that
\[a_2(T)=s_2(T) + T^{c_2}r_2(T)\]
has no terms of positive degree.  Setting $k_2(T)=1 + s_2(T)p^2$
we see that $k_2(T)k_1(T)k_0(T)L(T)$ reduces modulo $p^3$ to an element of 
$\mathcal{O}_K/p^3 \mathcal{O}_K[T,T^{-1}]$.  Continuing this process inductively 
we see that \[B(T)=\prod_{i=0}^\infty k_i(T)\] satisfies the desired properties.

\end{proof}

\bibliographystyle{plain}
\bibliography{log_growth_rings.bbl}

\Addresses

\end{document}